\documentclass{amsart}
\usepackage{epsfig}

\newtheorem{theorem}{Theorem}[section]
\newtheorem{prop}[theorem]{Proposition}
\newtheorem{lemma}[theorem]{Lemma}
\newtheorem{cor}[theorem]{Corollary}

\theoremstyle{definition}
\newtheorem{definition}[theorem]{Definition}
\newtheorem{example}[theorem]{Example}

\theoremstyle{remark}
\newtheorem{remark}[theorem]{Remark}

\numberwithin{equation}{section}

\begin{document}

\title{Invariants of stable maps from the $3$-sphere to the Euclidean $3$-space}

\author{N. B. Huaman\'{i}}
\address{Instituto de Matem\'atica y Ciencias Afines (Imca-Uni). Calle Los Bi\'ologos 245.
Lima 15012
Per\'u.}
\email{nelson.berrocal@imca.edu.pe}
\thanks{The first author was supported in part by  Fondecyt C.G. 176-2015.}


\author{C. Mendes de Jesus}
\address{Departamento de Matem\'atica, Universidade
\ Federal de Vi\c cosa,  \ 36570-000, Vi\c cosa - MG, Brasil.}
\email{cmendes@ufv.br}

\author{J.  Palacios}
\address{Instituto de Matem\'atica y Ciencias Afines (Imca-Uni). Calle Los Bi\'ologos 245.
Lima 15012
Per\'u.}
\email{jpalacios@imca.edu.pe}

\subjclass[2010]{ 57R45, 58K15, 58K65.}

%
%
%



\keywords{Stable maps, singular sets, branch sets, $3$-sphere.}

\begin{abstract}
In the present work, we study the decompositions of codimension-one transitions that alter the singular set the of  stable maps of $S^3$ into $\mathbb{R}^3,$ the topological behaviour of the singular set and the singularities in the branch set that involves cuspidal curves and swallowtails that alter the singular set. We also analyse the effects of these decompositions on the global invariants with prescribed branch sets.
\end{abstract}

\maketitle
\section*{Introduction}
The study of stable maps between manifolds is an ongoing research topic that has interested many mathematicians around the world \cite{Arnold, Eliasberg, Goryunov, Tamas, Ohmoto,  Pignoni, Saeki, Yamamoto}. In the case of stable maps from $3$-manifolds to the Euclidean space $\mathbb{R}^3$, Goryunov, in his article \textit{Local invariants of maps between 3-manifolds}  \cite{Goryunov}, classifies first-order invariants of maps between 3-manifolds whose increments in generic homotopies are defined entirely by diffeomorphism types of local bifurcations. This local study gives a step towards the global study of stable maps from compact, oriented $3$-manifolds to the space $\mathbb{R}^3$. Mendes de Jesus, Shina and Romero-Fuster \cite{MSR} introduced weighted graphs associated to stable maps, as global topological invariants, in parallel with the study of maps between surfaces done by Hacon, Mendes de Jesus and Romero-Fuster in \cite{HMR1, HMR3}. The aforementioned authors studied the global point of view of stable maps from closed surfaces to the plane, and provided techniques for constructing maps, as well as introducing graphs with weights on its vertices associated to stable maps between surfaces, as global invariants. Moreover, these invariants improve the ones determined by Aicardi and Ohmoto \cite{Ohmoto} classifying stable maps from closed surfaces to the two-dimensional plane from a global point of view. These works are based on Arnold's work \cite{Arnold}, where he introduced invariants for stable embeddings of $S^1$ into $\mathbb{R}^2$, using the techniques of \textit{Vassiliev's invariants} \cite{Vassiliev}. In parallel,  Yamamoto \cite{Yamamoto} determined  semi-local invariants of stable maps of 3-manifolds to the plane $\mathbb{R}^2$. In  \cite{SR,SR2}, Sinha and Romero-Fuster obtained other results, from both the local and global point of view, for maps from $3$-manifolds to the space $ \mathbb{R}^3$, also paying attention to the fold maps.

The main objective of this work is to study in more detail the stable maps from $S^3$ to $\mathbb{R}^3$, aiming to complement this particular case to the results  given in \cite{MSR, SR}. More precisely, we intend to describe the topological behaviour of the singular set and the singularities in the branch set, that involves the components of a singular set, cuspidal curves and swallowtails (see Definition \ref{def01}). We study the decompositions of the codimension-one transitions, listed in \cite{Goryunov}, that alter the singular set, and present some results related to the effects of these decompositions on the global invariants of maps from the $3$-sphere to the $3$-dimensional Euclidean space.  

In Section \ref{sec2}, we present a summary of the concepts of stable maps from a $3$-manifold to $\mathbb{R}^3$, with the definition of five global invariants of stable maps. In Section \ref{sec3}, we introduce the decomposition of codimension-one transitions and their respective properties from a global point of view. In Section \ref{sec4}, we present the relationship between the above global invariants and its consequences on the construction of stable maps and fold maps with a prescribed singular set.  

The main result of this work is Theorem \ref{teoA} and its consequences.

\section{Stable maps of $3$-manifolds into $\mathbb{R}^3$}
\label{sec2}

Unless otherwise specified, a {\em manifold} will mean a smooth manifold and a {\em map} between manifolds will be a smooth map.  

Let $M$ be a $3$-manifold. The set $C^{\infty} (M,\mathbb{R}^3)$ is the space of (smooth) maps from $M$ to $\mathbb{R}^3$.  Two maps $f,g \in C^{\infty} (M,\mathbb{R}^3)$   are {\it $\mathcal{A}$-equivalent} whenever there are two diffeomorphisms $\phi \colon M\longrightarrow M$ and $\psi\colon \mathbb{R}^3 \longrightarrow \mathbb{R}^3$ such that $ g = \psi \;\circ\; f \;\circ\; \phi^{-1}$. 
A map  $f \in C^{\infty} (M,\mathbb{R}^3)$ is {\it stable}, if every map sufficiently close to $f$ (in the Whitney $C^{\infty}$-topology) is equivalent to $f$. Acording to Whitney \cite{W},  the set of all stable maps, denoted by $\mathcal{E}(M,\mathbb{R}^3)$,  is open and dense in $C^{\infty} (M, \mathbb{R}^3)$.  A point  $x$ in $M$ is  {\it  regular} with respect to a map $f$ in $C^{\infty} (M,\mathbb{R}^3)$  if $f$ is a local difeomorphism around a neighborhood of $x$, otherwise $x$ is called a \textit{singular point}. If $f$ is a generic map, then its singular set has codimension $2$ and it is formed by submanifolds of dimensions zero, one and two, see \cite{ROS}. In \cite{Gib},  it is shown that the normal forms of the germs at the singular points of a stable map $f$ are the following:
\begin{itemize}
\item[a)]  $A_{1}$:   a \textit{double point},  $(x,y,z)\mapsto (x^2,y,z)$;
\item[b)] $A^{\pm}_{2}$; a  \textit{cusp point},  $(x,y,z)\mapsto (\pm x^{3}+yx,y,z)$;
\item[c)] $A^{\pm}_{3}$:  a \textit{swallowtail point}, $(x,y,z)\mapsto (\pm x^{4}+yx^{2}+zx,y,z)$.
\end{itemize}
If $M$ is a closed, oriented $3$-manifold,
then the \textit{singular set} of $f$, denoted by $\Sigma f$,  is formed by the following: ($i$) a disjoint union of closed, oriented surfaces embedded in $M$; ($ii$) curves consisting of cusp points, in which isolated swallowtail points might exist; and ($iii$) fold points. The surfaces in ($i$) separate the regular   components of $f$ whose boundaries are contained in $\Sigma f$. 
\begin{figure}[htp]
\vspace{-0.3cm}
$$ \epsfxsize=11cm \epsfbox{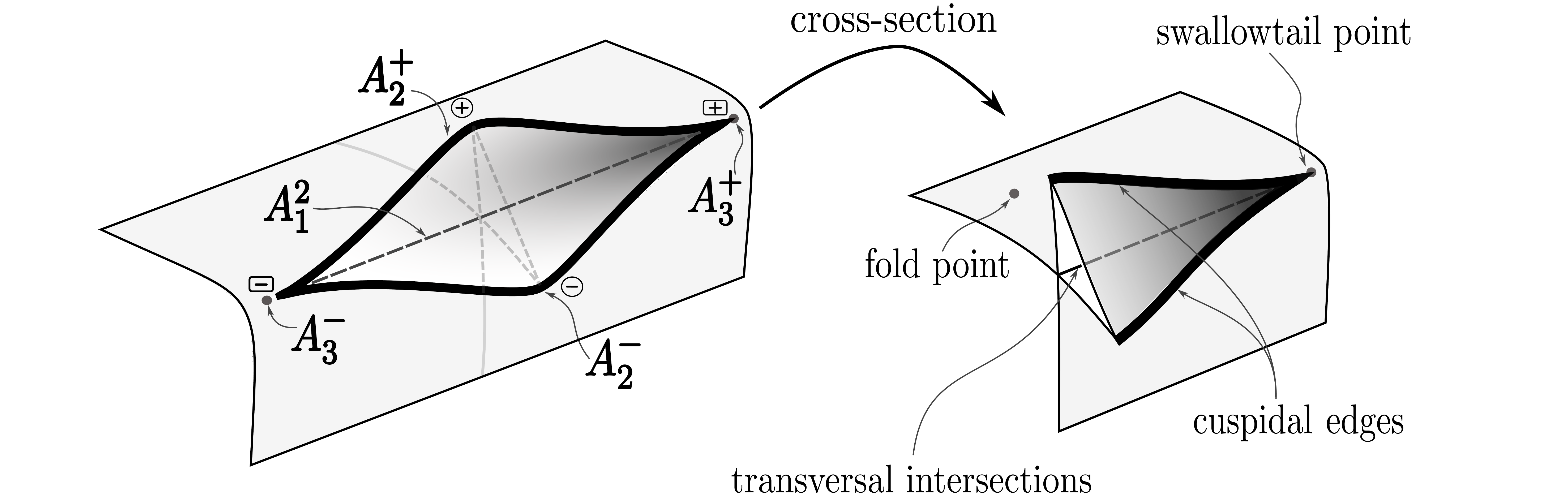} $$
\vspace{-0.8cm}
\caption{Examples of cuspidal and swallowtail curves. \label{coladegolondrina2}}
\end{figure}
The \textit{branch set} of $f$ is the image $f(\Sigma f)$ of the singular set. It is formed by a collection of closed, oriented surfaces embedded in $\mathbb{R}^3$ possibly with transversal intersections and singularities corresponding to finitely many cuspidal edges and isolated swallowtail points 
 (see Figure \ref{coladegolondrina2}).  The branch set of $f$ have the following auto-intersections (see Figure \ref{Intersection1}):
 
\begin{itemize}
\item[1.]  $A^{2}_{1}$: transversal intersection  of two smooth sheets,
\item[2.] $A^{\pm}_{2} A_{1}$:  transversal intersection of a cuspidal edge with a regular sheet,
\item[3.] $A^{3}_{1}$:  isolated threefold points obtained by the transversal intersection of three sheets.
\end{itemize}
\begin{figure}[htp]
\vspace{-0.63cm}
$$ \epsfxsize=10.9cm \epsfbox{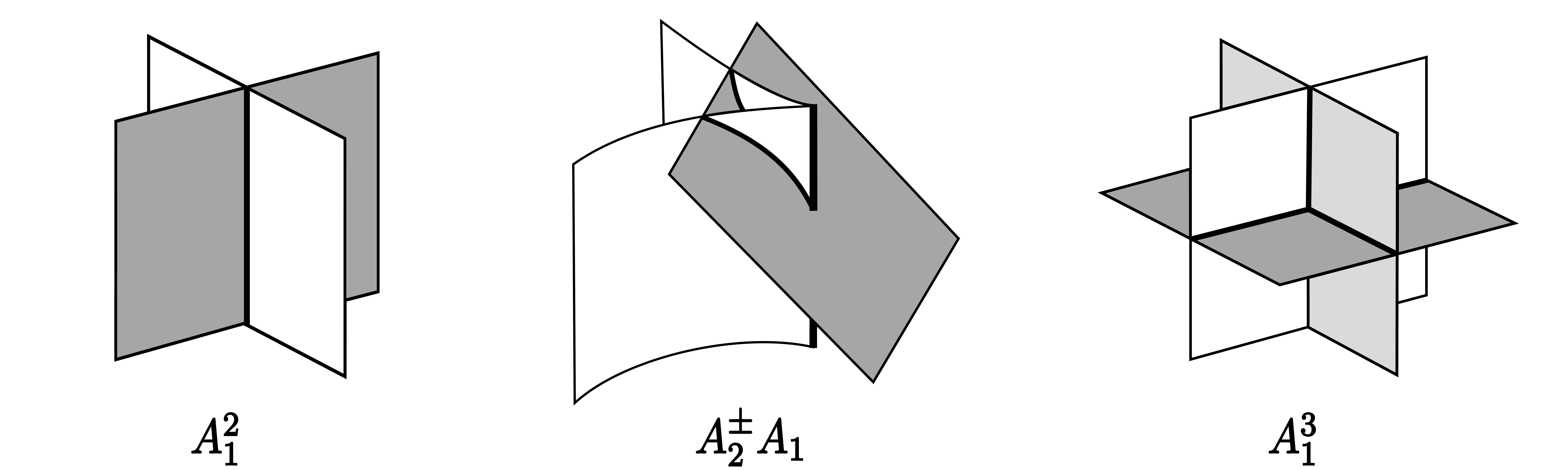} $$
\vspace{-0.9cm}
\caption{Points of type $A^{2}_{1},\;A^{\pm}_{2} A_{1}$ and $A^{3}_{1}$  . \label{Intersection1}}
\end{figure}
Suppose that the image $f(S_i)$ of a singular surface $S_i\subset \Sigma f$ is a compact connected surface embedded in $\mathbb{R}^3$. One might associate it with an {\em inward} (resp. {\em outward direction}) if the image of regular points in a sufficiently small tubular neighbourhood $\mathcal{V}$ of $S_i$ in $M$ such that $\mathcal{V}\cap \Sigma f =S_i$ are in the inner (resp. outer) region of $f(S_i)$. These directions will be represented by an inward or outward segment perpendicular to $f(S_i)$ according to the case, see Figure \ref{apliS3tor3E3m}.

 For example, let us consider a canonical projection $\pi\colon S^3\longrightarrow \mathbb{R}^3$, so that  $\Sigma \pi=S^2$ and its image $\pi(\Sigma \pi)=\pi(S^2)=S^2$ is also a $2$-sphere. Since the image of the regular points lies in the inner region of $S^2$, it has an inward direction.  
\begin{definition}\label{def01}
A \textit{cuspidal curve} is a closed curve consisting of cusp points (see $f_2$ and $ f_3 $ in Figure \ref{apliS3tor3b}). 
\end{definition}

If $f$ is a stable map from a compact, oriented $3$-manifold $M$ to $\mathbb{R}^3$, then the number of singular surfaces in $M$, cuspidal curves, and swallowtail points in $f(\Sigma f)$ are all finite. This motivates the following:
\begin{definition} 
Let $f$ be a stable map from a compact, oriented $3$-manifold $M$ to $\mathbb{R}^3$. The {\em global invariants} of $f$ are: 
\begin{itemize}
\item[]$I_E(f):$ the number of singular surfaces, 
\item[]$I_V(f):$ the number of regular components, 
\item[]$I_C(f):$ the number of cuspidal curves,
\item[]$I_G(f):$ the sum of the genus of all singular surfaces,
\item[]$I_S(f):$ the number of swallowtail points.
\end{itemize}
Sometimes we write $I_E$ instead of $I_E(f)$ if no confusion arises, similarly for the other invariants.
\end{definition}

If $M=S^3$, then $I_V=I_E+1$, which is a consequence of the Jordan-Brouwer theorem, since $\Sigma f$ is a disjoint union of closed, oriented surfaces embedded in  $S^3$. Therefore, we shall only consider the invariants $I_E,I_C,I_G,I_S$.
\begin{figure}[htp]
\vspace{-0.35cm}
$$ \epsfxsize=11.8cm \epsfbox{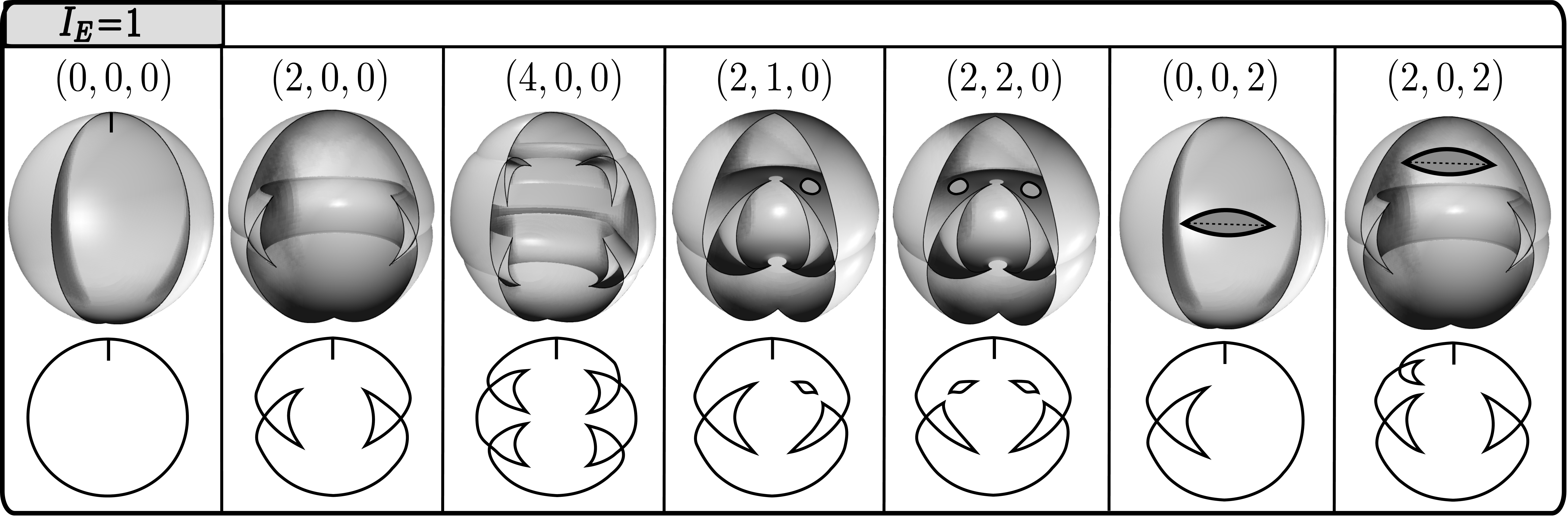} $$
\vspace{-0.6cm}
$$ \epsfxsize=11.8cm \epsfbox{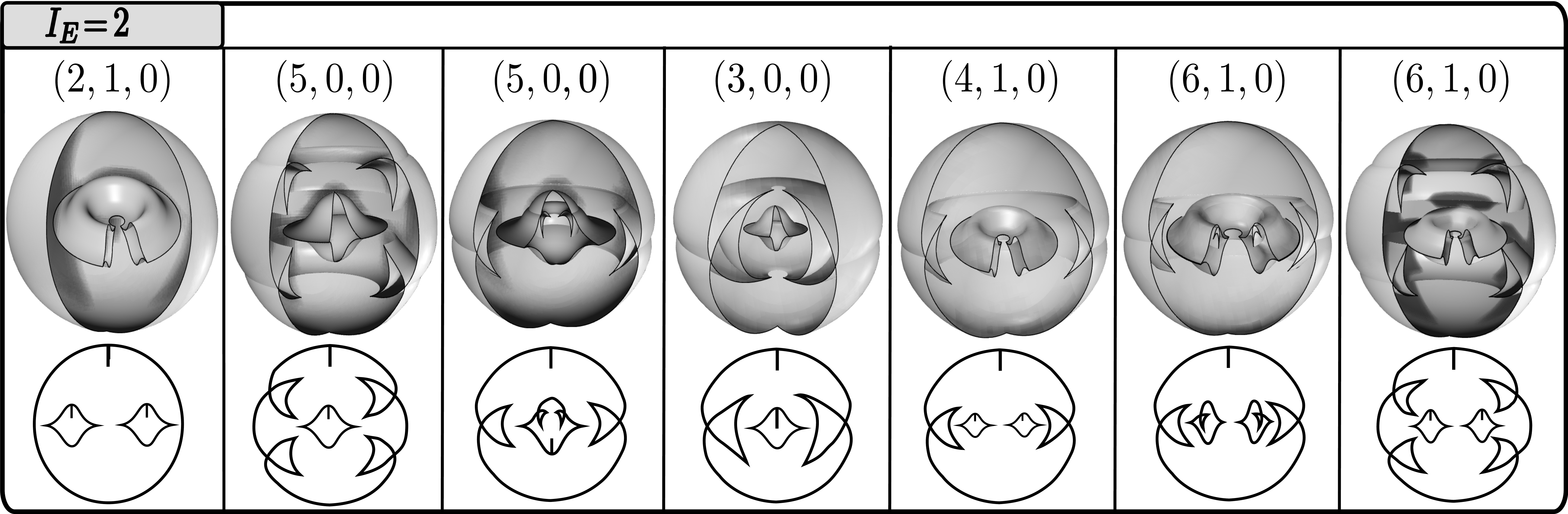} $$
\vspace{-0.5cm}
$$ \epsfxsize=11.8cm \epsfbox{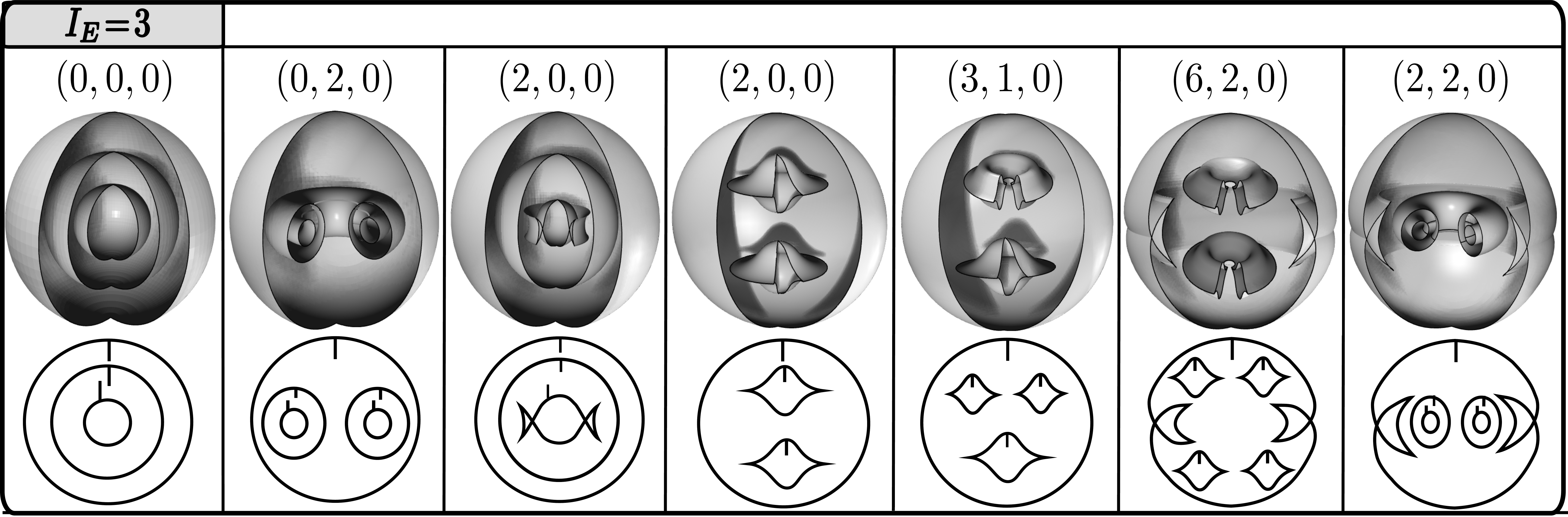} $$
\vspace{-0.8cm}
\caption{Examples of stable maps from $S^3$ to $\mathbb{R}^3$ with $I_E\leq 3$. \label{apliS3tor3E3m}}
\end{figure}
\begin{example} 
The Figure \ref{apliS3tor3E3m} shows the branch set of stable maps from $S^3$ to $\mathbb{R}^3$ with $I_E\leq 3$, in which the triplet of integers means the triplet $(I_C,I_G,I_S)$. For $I_E=1$, the stable maps have $I_C\leq 4$, $I_G\leq 2$ and $I_S\leq 2$;
for $I_E=2$, the stable maps have  $I_C\leq 6$, $I_G\leq 1$  and $I_S=0$; and
for  $I_E=3$, the stable maps have $I_C\leq 6$, $I_G\leq 2$ and $I_S=0$.
\end{example}

\section{Codimension-one transitions}
\label{sec3}
The complement of the set $\mathcal{E}(M,\mathbb{R}^3)$ in  $C^{\infty} (M,\mathbb{R}^3)$ is called {\em discriminant set}.  
Let us consider a homotopy $F\colon M\times [a,b]\longrightarrow \mathbb{R}^3$, $(x,t)\mapsto F(x,t)=F_{t}(x)$, between two stable maps $f,f'\colon M\longrightarrow \mathbb{R}^3$. 
As $t$ varies in $[a,b]$, the branch set of $F_{a}=f$ is continuously deformed into the branch set of $F_{b}=f'$. A map $F_{t_{0}}$ may lie in the discriminant set for a certain $t_{0}\in [a,b]$. For instance, if $f$ and $f'$ are $\mathcal{A}$-equivalent, then there does not necessarily exist a time $t_{0}\in [a,b]$ such that $F_{t_{0}}$ lies in the discriminant set. However, if $f$ and $f'$ are not $\mathcal{A}$-equivalent, then there exists at least a time $t_{0}\in [a,b]$ such that $F_{t_{0}}$ lies in the discriminant set.  

Every map in $\mathcal{E}(M,\mathbb{R}^3)$ has codimension $0$ in the space $C^{\infty} (M,\mathbb{R}^3)$, and the discriminant set is formed by unstable maps of codimension greater than or equal to $1$, see \cite{Goryunov}. 
Goryunov lists the codimension-one transitions in Figure $3$ of {\em loc.cit.} (see Figure \ref{transicionesoriginales1}). We shall see in the next section that some of these transitions alter the topology of the singular set, whereas the others alter the number of cuspidal curves and swallowtail points. Every homotopy class of maps in $C^{\infty}(M,\mathbb{R}^3)$ is path-connected, hence there is a continuous path in $C^{\infty}(M,\mathbb{R}^3)$ joining two maps that are in two different classes of $\mathcal{A}$-equivalences and meeting the discriminant set through finitely many codimension-one maps, called {\em codimension-one transitions}. 
\begin{figure}[htp]
\vspace{-0.3cm}
$$ \epsfxsize=12cm \epsfbox{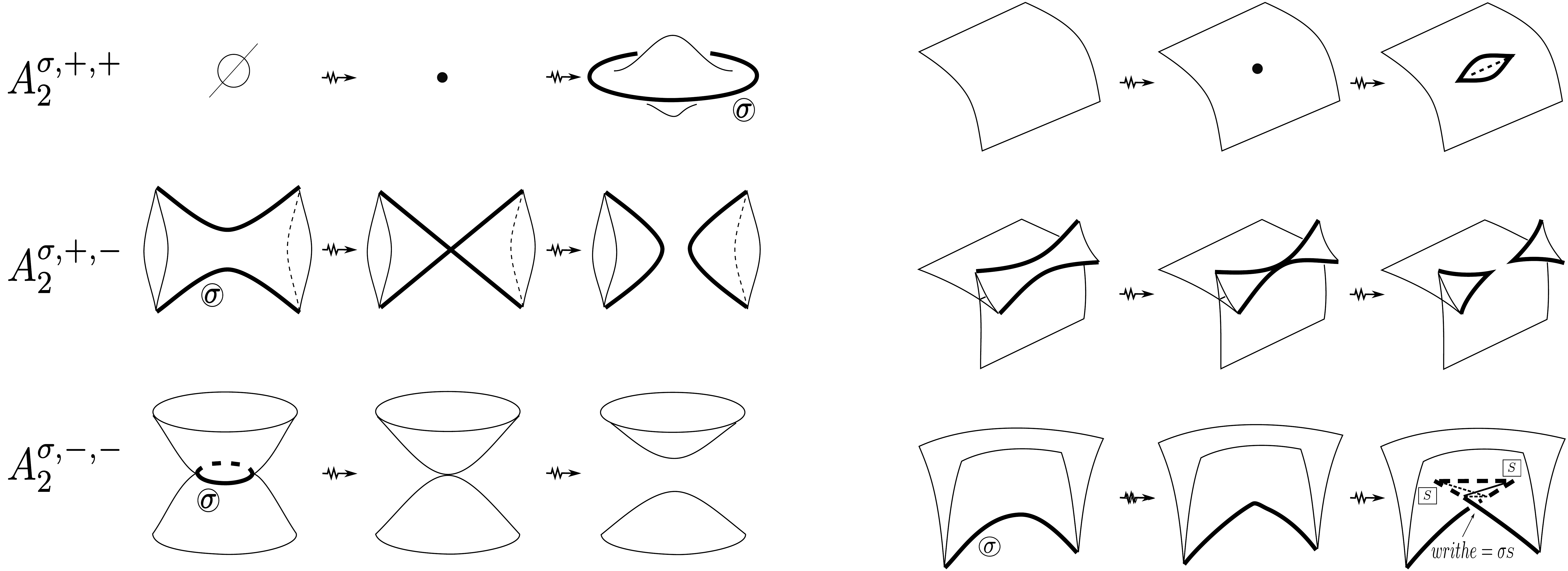} $$
\vspace{-0.8cm}
\caption{Codimension-one transitions. \label{transicionesoriginales1}}
\end{figure}
Figure \ref{transicionesoriginales1} illustrates the transitions $A^{\sigma,+,+}_{2},$ $ A^{\sigma,+,-}_{2}$, $ A^{\sigma,-,-}_{2}$, $A^{e}_3$ and $A^{h}_3$ described in \cite{Goryunov}, which are the transitions altering the number of cuspidal curves, and they might also alter  the topology of the singular set and the number of swallowtail points.  
For convenient purposes, the three first transitions that alter the topology of the singular set will be denoted by $L$, $B$ and $P$, respectively.
We shall see in more detail the global properties of the transitions $L,B,P,A^{e}_3,A^{h}_3$, and subdivide each one of them if necessary in order to distinguish the global properties that may happen when a path joining two stable maps goes through these transitions.  

\begin{definition} 
 Let $f$ be a map obtained from a map $f'$ after passing through a codimension-one transition  
 $T\in \{L, B , P,A^e_3, A^h_3\}$. 
 A transition $T$ has {\em positive} (resp. {\em negative}) direction, in the following cases:

\begin{itemize}
\item[$L$] has {\em positive} (resp. {\em negative}) direction, if $f$ has exactly one cuspidal curve more (resp. less) than $f'$.

\item[$B$] has {\em positive} (resp. {\em negative}) direction, if $f$ has exactly a cuspidal curve and one singular surface more (resp. less) than $f'$; or if $f$ has exactly one cuspidal curve less (resp more) than $f'$, and the difference between the genus of singular surfaces of $f$ and $f'$ is $-1$ (resp. $1$). 

\item[$P$] has {\em positive} (resp. {\em negative}) direction, if $f$ has exactly one 
cuspidal curve less (resp. more) than $f'$.  

\item[$A^e_3,$]$A^h_3$  have {\em positive} (resp. {\em negative}) direction, if $f$ has two swallowtails more than $f'$. 
\end{itemize}
\end{definition}
\noindent
In the next paragraphs, we shall describe the behaviour and decomposition of transitions according with the alteration of the number of cuspidal curves, swallowtails and singular surfaces in positive direction, see Figure \ref{transeje1b}.
\vspace{0.5cm}

\noindent
{\boldmath $ L $} : This transition will not be decomposed. It creates a singular surface homeomorphic to the sphere, and the number of singular surface and the number of cuspidal curves both increase by one. We denote by $Q$ and $\eta$ the new singular surface and cuspidal curve, respectively (see Figure \ref{transL}). Thus,  
$$(I_E, I_C, I_G, I_S)(f)=(I_E, I_C, I_G, I_S)(f')+(1,1,0,0)\,.$$
\begin{figure}[htp]
\vspace{-0.3cm}
$$ \epsfxsize=12.3cm \epsfbox{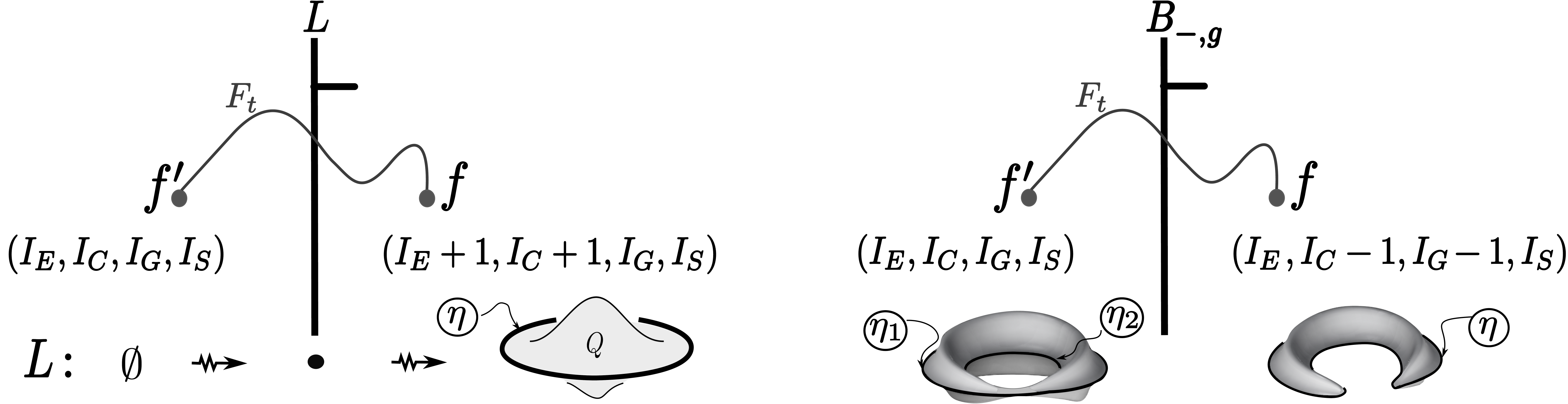} $$
\vspace{-0.8cm}
\caption{Example of transitions $L$ and $B_v$. \label{transL}}
\end{figure}
\noindent
{\boldmath $B$}: It modifies not only the number of cuspidal curves but the number of the singular surfaces and their genera. It will be decomposed into 
 $B_{+,g}$, $B_{0,g}$, $B_{-,g}$ and $B_v$ such that $B=B_{+,g}+B_{0,g}+B_{-,g}+B_v$. 
Figure \ref{transeje1b} shows a local picture of the transitions in the branch set with respect to the positive direction, where:
\begin{itemize}
\item[$\bullet$] {\boldmath $ B_{-,g} $}: The cuspidal curves $\eta_1$ and $\eta_2$ in the surface $W$ join each other tangentially to become a cuspidal curve $\eta$ in a surface $Z$,  the number of cuspidal curves and the genus of a singular surface decrease by one (see Figure \ref{transL}). Thus,    
$$(I_E, I_C, I_G, I_S)(f)=(I_E, I_C, I_G, I_S)(f')+(0,-1,-1,0).$$
\end{itemize}
\begin{itemize}
\item[$\bullet$] {\boldmath $ B_{0,g} $}: Two arcs of the cuspidal edges $\eta_1$ and $\eta_2$ in the surface $W$ join each other tangentially and split into two cuspidal edges $\beta_1$ and $\beta_2$ in a surface $Z$,  the number of the genus of a singular surface decrease by one. Thus,   
$$(I_E, I_C, I_G, I_S)(f)=(I_E, I_C, I_G, I_S)(f')+(0,0,-1,0).$$
\item[$\bullet$]{\boldmath $ B_{+,g} $} : Two arcs of the cuspidal curve $\nu$ in the surface $W$ join each other tangentially and split into two cuspidal curves $\nu_1$ and $\nu_2$ in a surface $F$,  the number of cuspidal curves increases by one and the genus of a singular surface decreases by one. Thus, 
$$(I_E, I_C, I_G, I_S)(f)=(I_E, I_C, I_G, I_S)(f')+(0,1,-1,0)\,.$$
\item[$\bullet$] {\boldmath $ B_{v} $}: Two arcs of the cuspidal curve $\nu$ in the surface $W$ join each other tangentially and split into two cuspidal curves $\nu_1$ and $\nu_2$ in the surfaces $U_1$ and $U_2$ respectively. The number of  singular surfaces and the number of cuspidal curves both increase by one. Thus, 
$$(I_E, I_C, I_G, I_S)(f)=(I_E, I_C, I_G, I_S)(f')+(1,1,0,0)\,.$$
\end{itemize}
\noindent
{\boldmath $ P $}: This transition modifies the number of cuspidal curves, singular curves and the genus of singular surfaces. It will be decomposed into $P_g$ and $P_v$ such that $P=P_g+P_v$, where:  
\begin{itemize}
\item[$\bullet$] {\boldmath $ P_g $} : It eliminates the cuspidal curve $\eta$ and hole in the surface $W$, obtaining a new singular surface $D$. Thus,  
$$(I_E, I_C, I_G, I_S)(f)=(I_E, I_C, I_G, I_S)(f')+(0,-1,-1,0).$$
\item[$\bullet$] {\boldmath $ P_v $} : It eliminates the cuspidal curve $\eta$ by shrinking it in order to decompose the surface $W$ into two new singular surfaces, obtaining new singular surfaces $K_1$ and $K_2$ (see Figure \ref{transBP}). Thus,   
$$(I_E, I_C, I_G, I_S)(f)=(I_E, I_C, I_G, I_S)(f')+(1,-1,0,0).$$
\end{itemize}
\begin{figure}[htp]
\vspace{-0.3cm}
$$ \epsfxsize=10cm \epsfbox{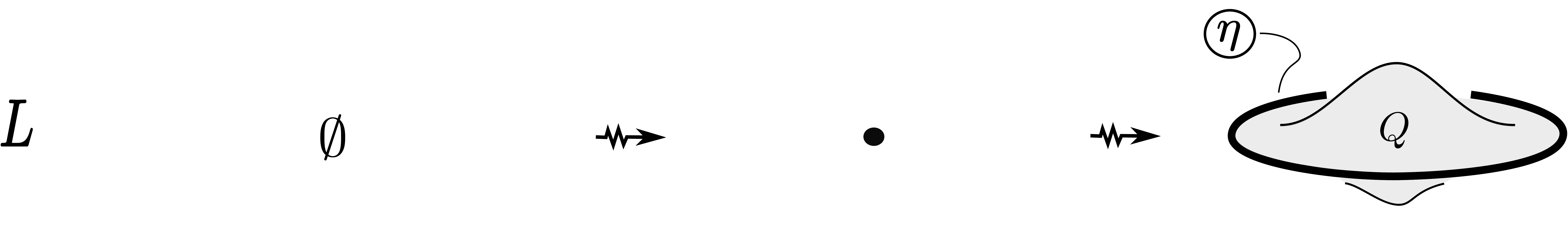} $$
\vspace{-0.44cm}
$$\epsfxsize=10cm \epsfbox{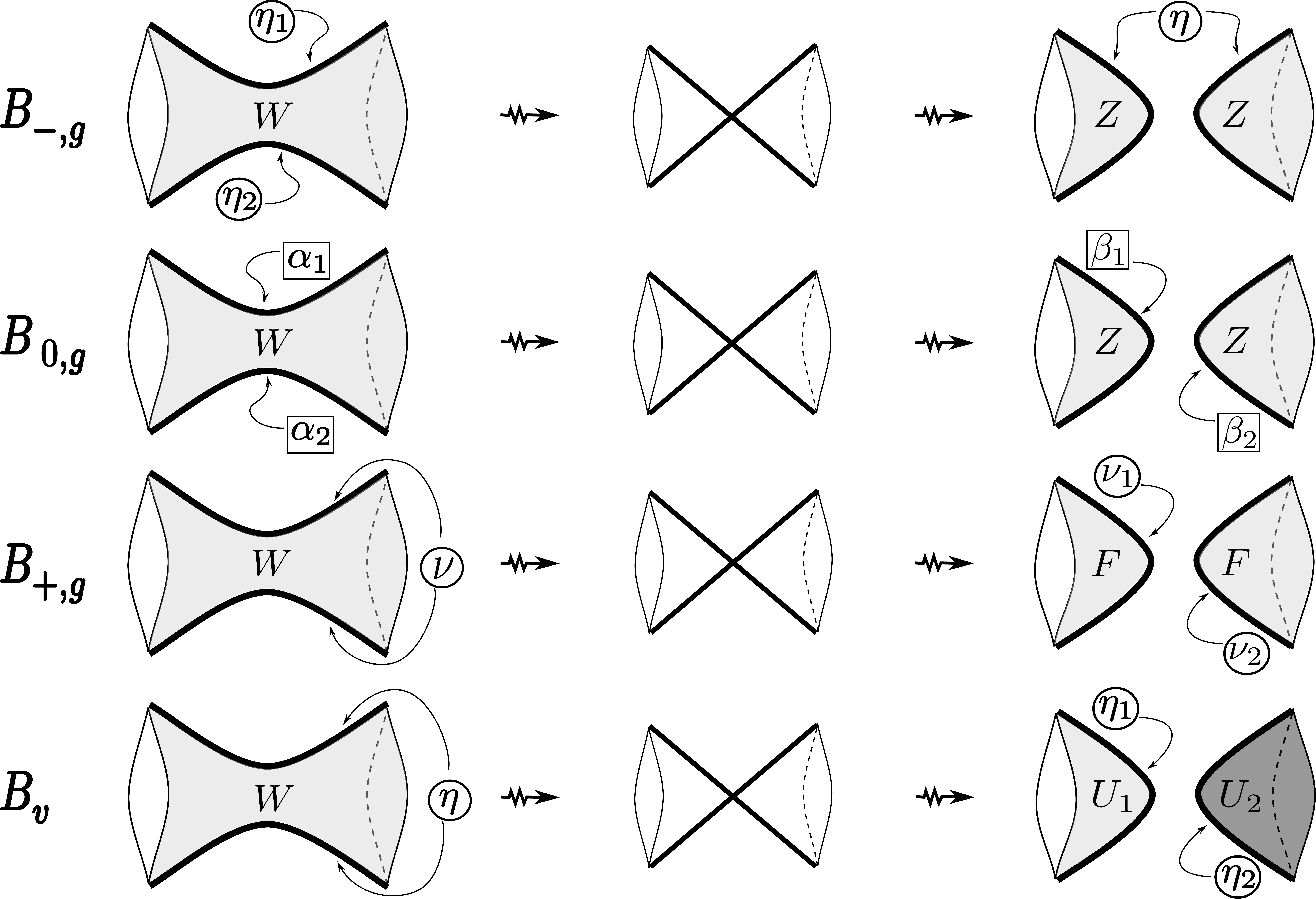} $$
\vspace{-0.44cm}
$$ \epsfxsize=10cm \epsfbox{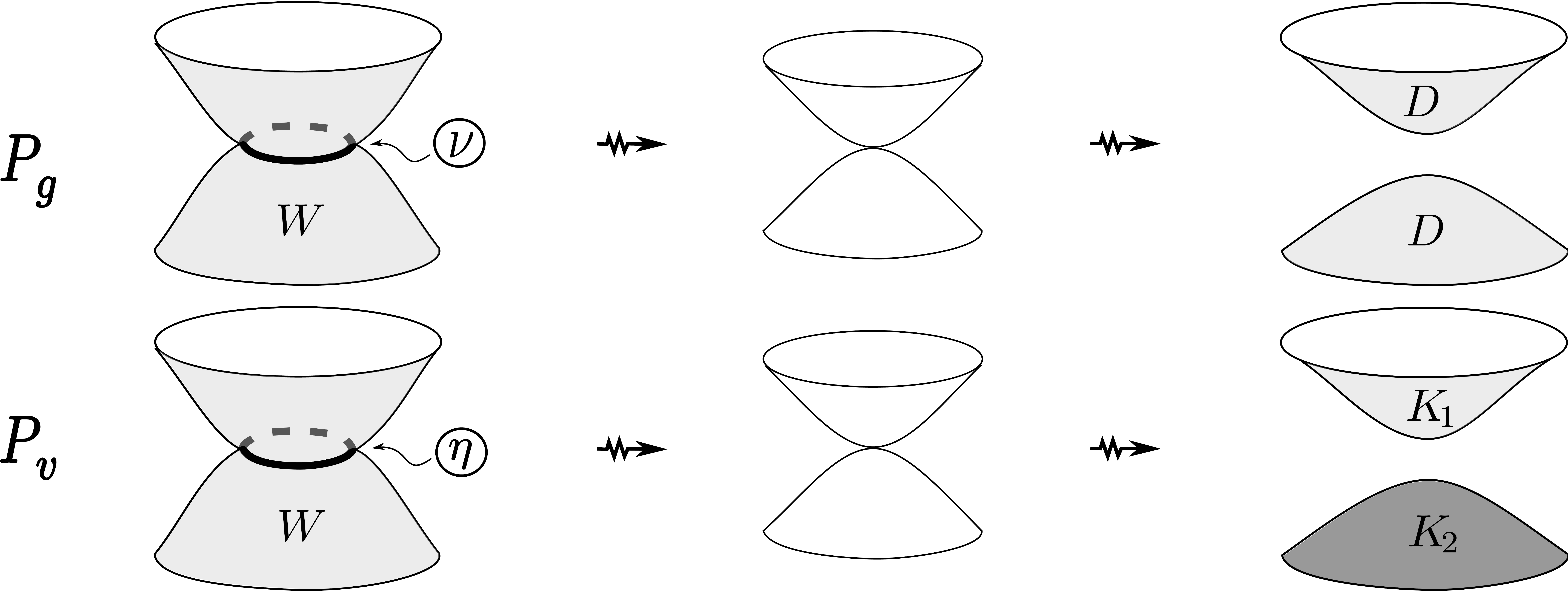} $$
\vspace{-0.44cm}
$$ \epsfxsize=10cm \epsfbox{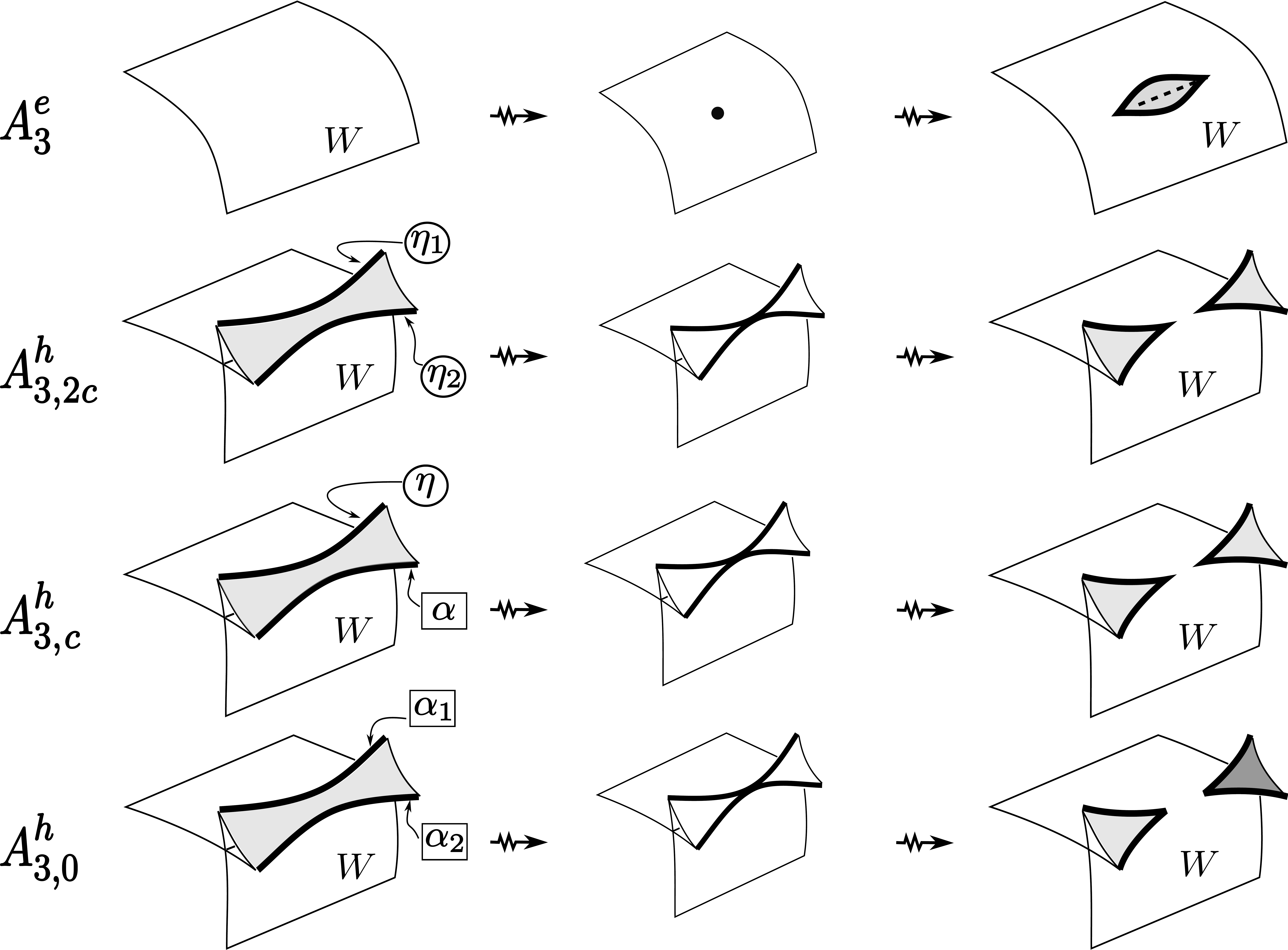} $$
\vspace{-0.7cm}
\caption{Decomposition of the transitions $B,$ $P$ and $A_{3}^{h}$. \label{transeje1b}}
\end{figure}
\noindent  {\boldmath $ A_{3}^{e} $} : This transition will not be decomposed. It creates two swallowtails in a region of a singular surfaces formed by fold points, the cuspidal lip gives birth to two new swallowtails. Thus, 
$$(I_E, I_C, I_G, I_S)(f)=(I_E, I_C, I_G, I_S)(f')+(0,0,0,2).$$
\noindent  {\boldmath $ A_{3}^{h} $} : This transition alters the number of swallowtails and the number of cuspidal curves. It will be decomposed into   
 $A_{3,2c}^{h},A_{3,c}^{h}$ and $A_{3,0}^{h}$ such that $A_{3}^{h}=A_{3,2c}^{h}+A_{3,c}^{h}+A_{3,0}^{h}$, where: 
\begin{itemize}
\item[$\bullet$]{\boldmath $ A_{3,2c}^{h} $} : Two arcs of the cuspidal curves $\eta_1$ and $\eta_2$ are joined tangentially in the surface $W$. The number of cuspidal curves decreases by two, and two new swallowtails are born. Thus, 
$$(I_E, I_C, I_G, I_S)(f)=(I_E, I_C, I_G, I_S)(f')+(0,-2,0,2).$$
\item[$\bullet$]{\boldmath $ A_{3,c}^{h} $} : An arc of a cuspidal curve $\eta$ and a cuspidal edge $\alpha$ are joined tangentially in the surface $W$. The number of cuspidal curves decreases by one, and two new swallowtails are born. Thus, 
$$(I_E, I_C, I_G, I_S)(f)=(I_E, I_C, I_G, I_S)(f')+(0,-1,0,2).$$
\begin{figure}[htp]
\vspace{-0.9cm}
$$ \epsfxsize=12cm \epsfbox{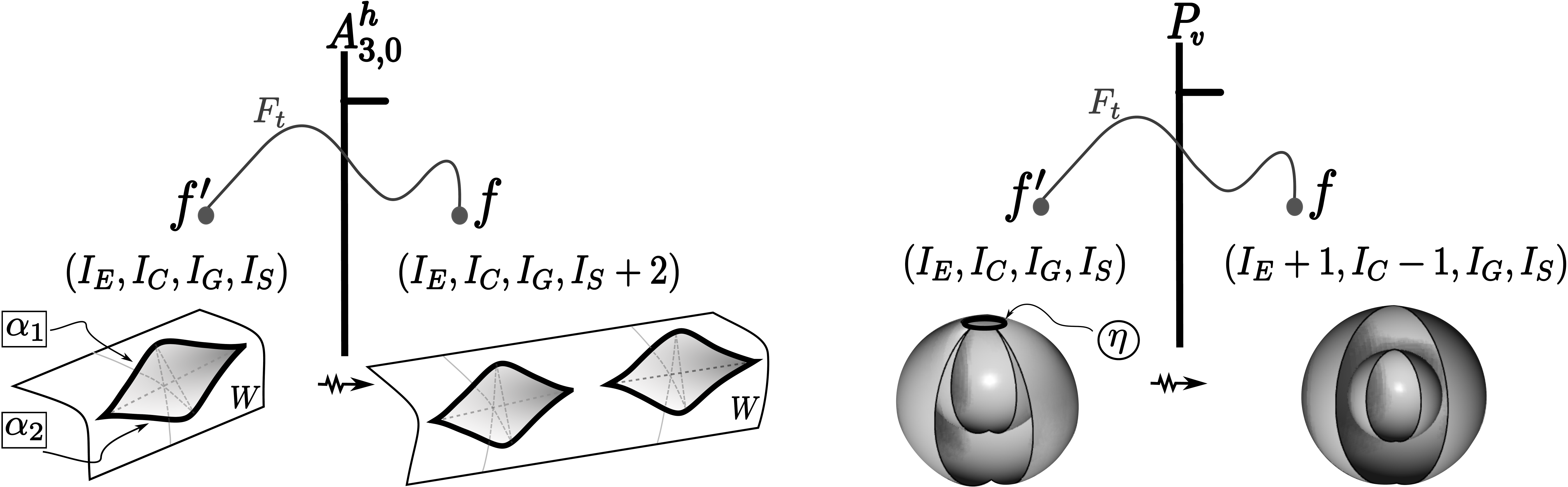} $$
\vspace{-0.8cm}
\caption{Example of transitions $A^h_{3,0}$ and $B_{-,g}$. \label{transBP}}
\end{figure}
\item[$\bullet$]{\boldmath $ A_{3,0}^{h} $} :  Two arcs of the cuspidal edges $\alpha_1$ and $\alpha_2$ are joined tangentially. The number of swallowtails increases by two (see Figure \ref{transBP}). Thus,  
$$(I_E, I_C, I_G, I_S)(f)=(I_E, I_C, I_G, I_S)(f')+(0,0,0,2).$$
\end{itemize}
A local picture of these three transitions is shown in Figure \ref{transeje1b}. 

We write 
\begin{equation}\label{transall}
\mathcal{T}=\{L, B_{-,g}, B_{0,g}, B_{+,g}, B_v , P_g, P_v,A^e_3, A^h_{3,2c}, A^h_{3,c}, A^h_{3,0}\}\,.
\end{equation}
The effects of the transitions in negative direction alter the cuspidal curves, singular surfaces and swallowtails exactly in the opposite way to what positive transitions do, as the case may be.
\begin{table}[htp]
\vspace{-0.3cm}
$$ \epsfxsize=12cm \epsfbox{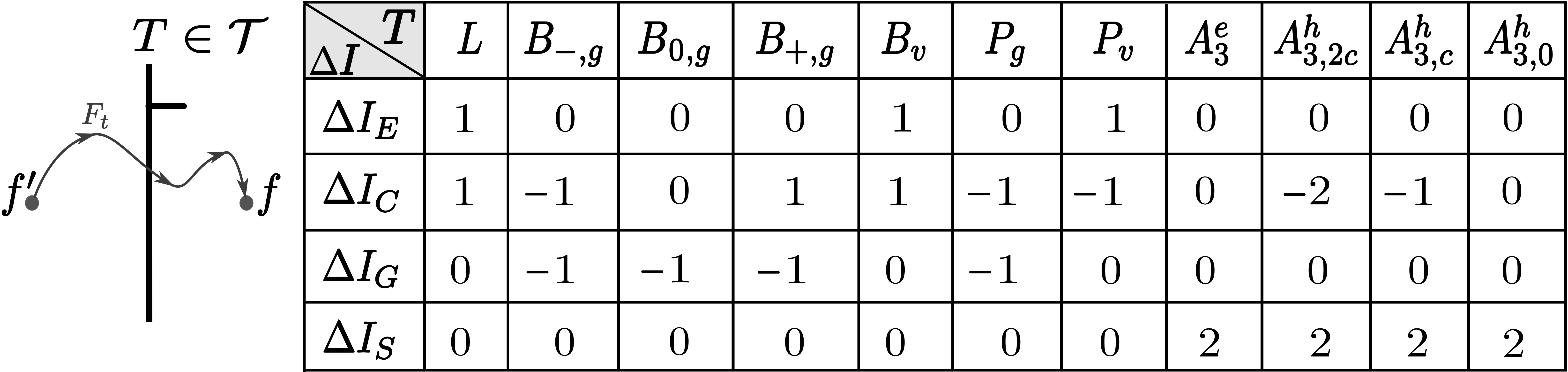} $$
\vspace{-0.35cm}
\caption{Increments of the invariants of the transitions $T\in \mathcal{T}$ in positive direction.\label{tablaLBPAeh} }
\end{table}
\vspace{-0.5cm}
\\
In Table \ref{tablaLBPAeh}, we collect the effects of all modifications of the invariants $I_E, I_C, I_G, I_S$ through the transitions in $\mathcal{T}$. If a map $f$ is obtained from a map $f'$, through a path $F_t$ crossing a transition $T$ in positive direction, we write $\Delta I_E=I_E(f)-I_E(f')$
for the increment of singular surfaces at the transition $T$. Similarly, we define the increments $\Delta I_C$, $\Delta I_G$ and $\Delta I_S$ at the transition $T$. On the other hand, $f'$ is obtained from $f$ through the path $F_{-t}$ crossing the transition $T$ in negative direction, in this case we have $\Delta I_E=I_E(f')-I_E(f)$. 
\begin{figure}[htp]
\vspace{-0.35cm}
$$ \epsfxsize=10.4cm \epsfbox{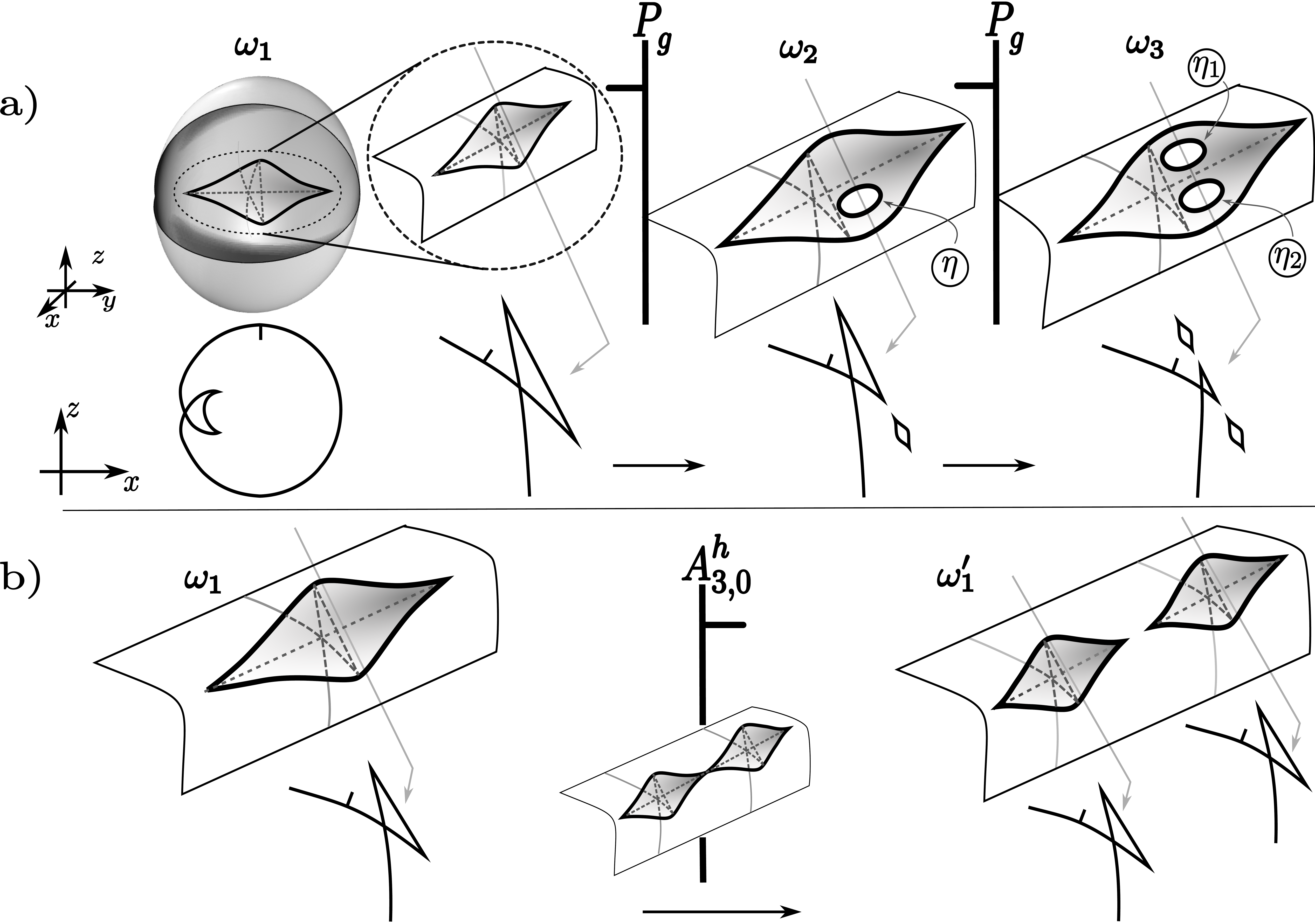} $$
\vspace{-0.6cm}
$$ \epsfxsize=10.4cm \epsfbox{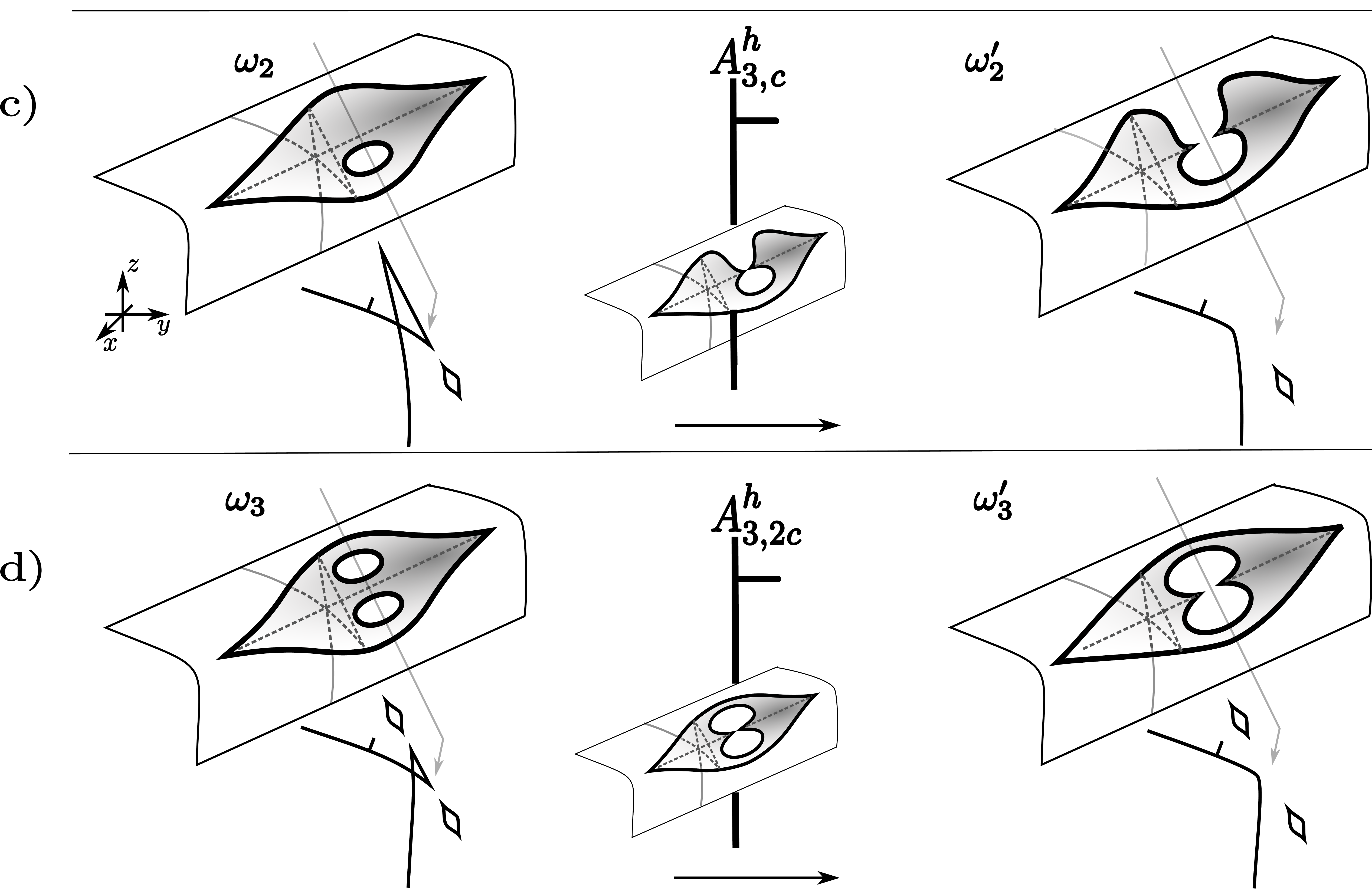} $$
\vspace{-0.8cm}
\caption{Example of local transitions  $P_g,$ $ A_{3,0}^{h},$ $A_{3,c}^{h} $ and $A_{3,2c}^{h}$. \label{transcolagolondrina1}}
\end{figure}

Figure \ref{transcolagolondrina1} shows the increments of the number of swallowtails  $I_S$ and cuspidal curves $I_C$. The curves underneath the branch sets are their cross-sections. Such curves have a perpendicular segment representing the direction of the singular surface. The horizontal arrows underneath the vertical segments indicate the positive direction of the transitions. One has the following: 
\begin{itemize}
\item[a)] The map $\omega_1$ is obtained from a canonical projection of $S^3$ onto $\mathbb{R}^3$, such that the path joining them crosses the transition $A^e_3$ in positive direction. The map $\omega_2$ is obtained from the map $\omega_1$, such that the path joining them crosses the transition $P_g$ in negative direction. Similarly, one obtains the map $\omega_3$ from the map  $\omega_2$. Then $I_C(\omega_3)=I_C(\omega_1)+2$ and $I_S(\omega_3)=I_S(\omega_1)$.
\item[b)] The map $\omega'_1$ is obtained from the map $\omega_1$, such that the path joining them crosses the transition $A^h_{3,0}$ in positive direction. Then $I_S(\omega'_1)=I_S(\omega_1)+2$.
\item[c)] The map $\omega'_2$ is obtained from the map $\omega_2$, such that the path joining them crosses the transition $A^h_{3,c}$ in positive direction. Then $I_C(\omega_2)=I_C(\omega'_2)+1$ and $I_S(\omega'_2)=I_S(\omega_2)+2$.
\item[d)] The map $\omega'_3$ is obtained from the map $\omega_3$, such that the path joining them crosses the transition $A^h_{3,2c}$ in positive direction. Then $I_C(\omega_3)=I_C(\omega'_3)+2$ and $I_S(\omega'_3)=I_S(\omega_3)+2$.
\end{itemize}

\begin{figure}[htp]
\vspace{-0.4cm}
$$ \epsfxsize=11.1cm \epsfbox{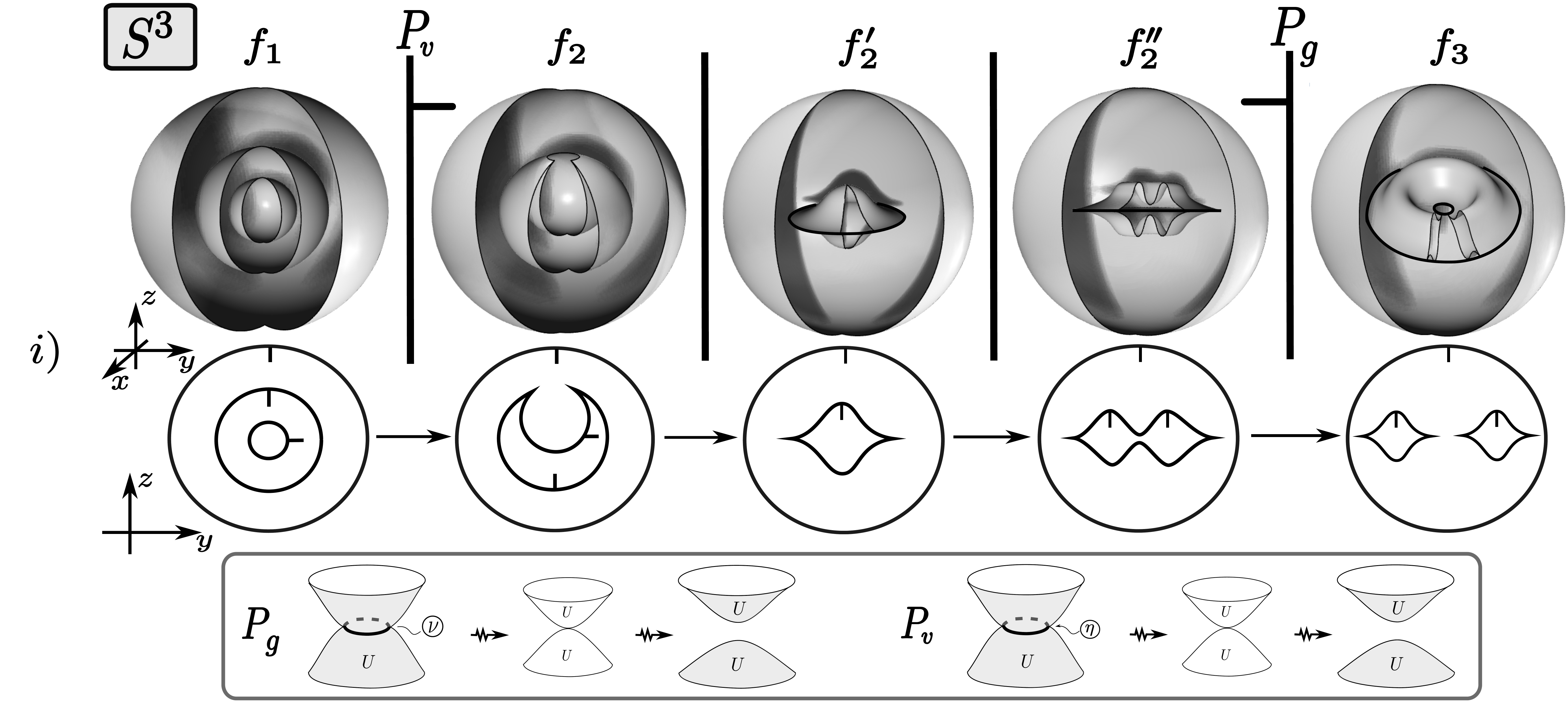} $$
\vspace{-0.47cm}
$$ \epsfxsize=11.1cm \epsfbox{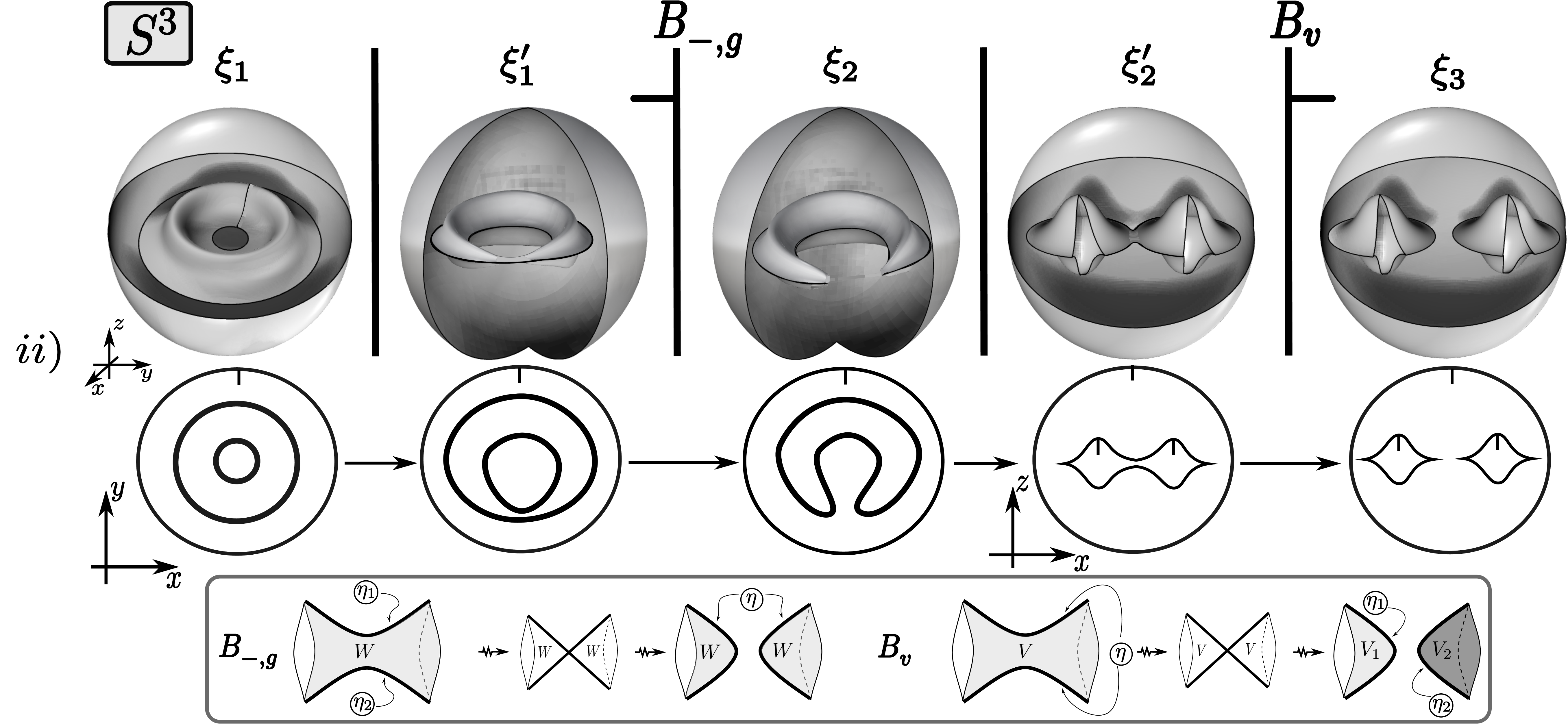} $$
\vspace{-0.47cm}
$$ \epsfxsize=11.1cm \epsfbox{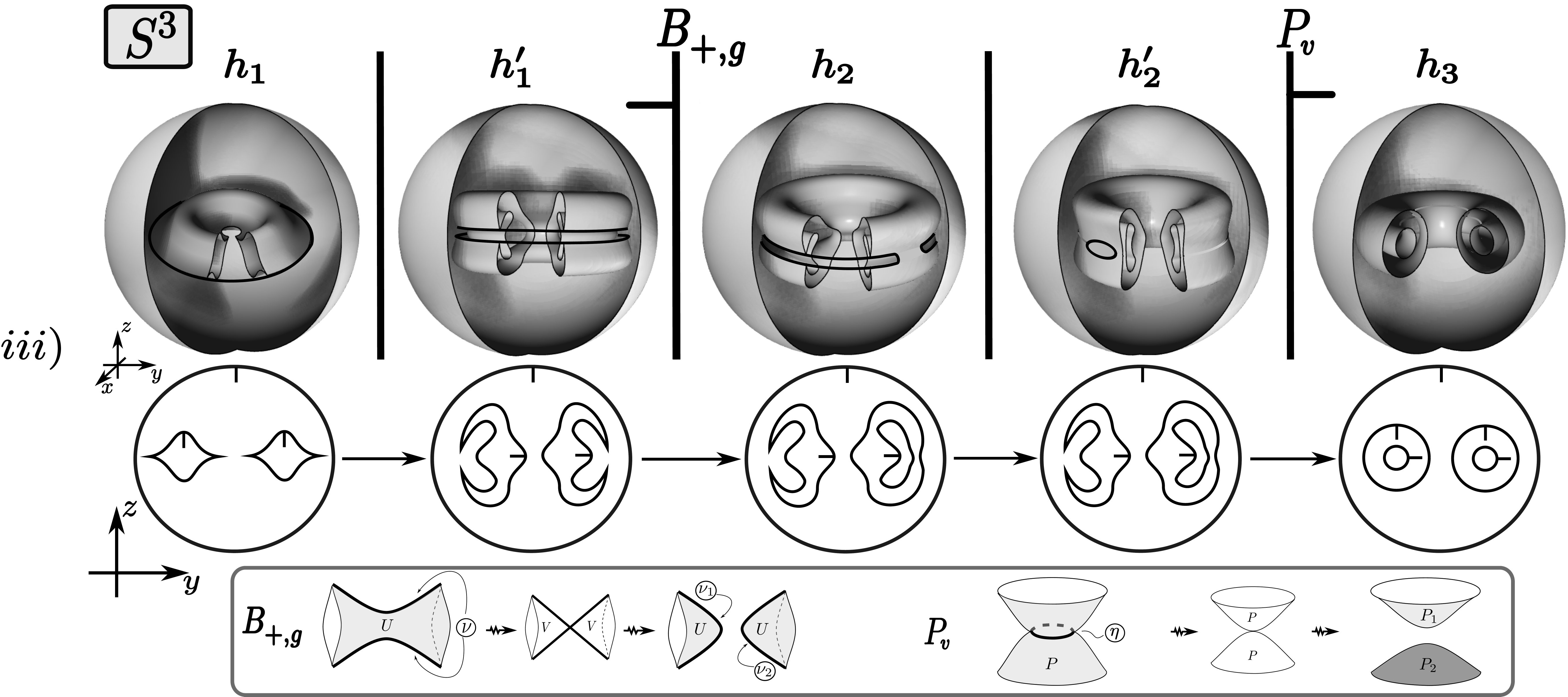} $$
\vspace{-0.7cm}
\caption{Example of transitions $ B_{+,g} $, $ B_{-,g} $, $B_v$, $P_g$ and $P_v$. \label{apliS3tor3b}}
\end{figure}
In \cite{MSR, SR}, the authors show that, if $ \Omega $ is union of a finite set of disjoint compact, oriented surfaces embedded in $S^3$, then there exists a stable map $f\colon S^3\longrightarrow \mathbb{R}^3$ such that  $\Sigma f=\Omega$. It is not straightforward to construct such maps. 

\vspace{0.3cm}
Figure \ref{apliS3tor3b} shows some examples of sequences of stable maps, in which the invariants  $I_E,I_C,I_G,I_S$ are less or equal than $3$, with respect to the pair $ (\Omega, S^3)$. Part $i)$ starts with the map $f_1$ such that $(I_E,I_C,I_G,I_S)(f_1)=(3,0,0,0)$. Now, $f_2$ is obtained from the map $f_1$ such that the path joining them crosses  the transition  $P_v$ in positive direction, so that we get $(I_E,I_C,I_G,I_S)(f_2)=(2,1,0,0)$. Hence, $f_3$ is obtained from the map $f_2$ such that the path joining them crosses the transition  $P_g$ in negative direction, so that we get $(I_E,I_C,I_G,I_S)(f_3)=(2,2,1,0)$. In the same way, we construct Part  $ii)$ and $iii)$ in Figure \ref{apliS3tor3b}.
\begin{table}[htp]
\vspace{-0.2cm}
$$ \epsfxsize=12.1cm \epsfbox{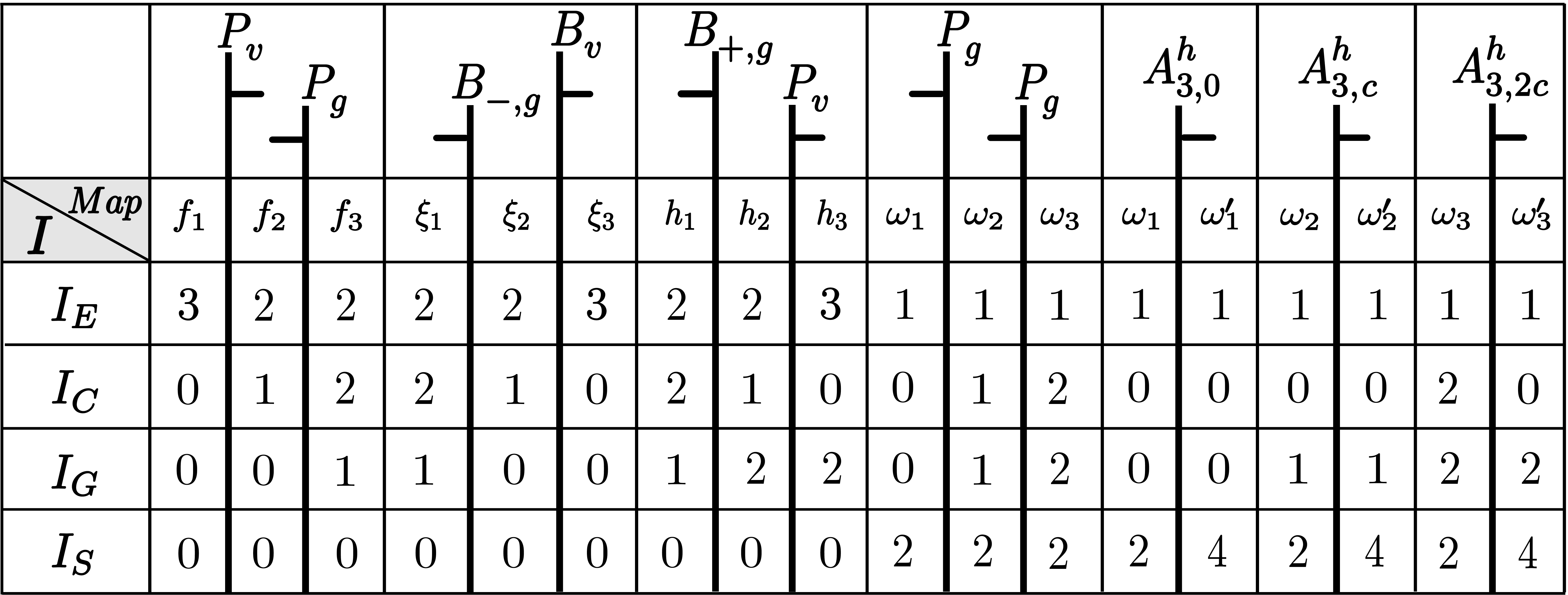} $$
\vspace{-0.3cm}
\caption{Summary of the transitions in Figure \ref{transcolagolondrina1} and \ref{apliS3tor3b}.\label{ejemploconstruccion} }
\end{table} 

\noindent
Table \ref{ejemploconstruccion} presents a summary of the effects of the transitions  $ B_{\pm,g} $, $B_v$, $P_g$  and $P_v$ in the invariants $I_E,I_C,I_G,I_S$, obtained from the maps constructed in Figure \ref{transcolagolondrina1} and Figure \ref{apliS3tor3b}.
\begin{prop} \label{propA}
Let  $\{S_i\}_{i=1}^{m}$ be a collection of $m$ disjoint surfaces embedded in  $S^3$ without self-intersections. Put $q=\sum_{i=1}^{m}g(S_i)$, where $g(S_i)$ is the genus of $S_i$. 
\begin{enumerate}
\item[ $(a)$] There exists a map $f\colon S^3\longrightarrow\mathbb{R}^3$ such that $\Sigma f=\bigcup_{i=1}^{m} S_i$, $I_C(f)= q+m-1$, $I_G(f)=q$ and $I_S(f)=0$.
\item[$(b)$] If $m=2n+1$, then there exists a map $f'\colon S^3\longrightarrow\mathbb{R}^3$ such that $\Sigma f'=\bigcup_{i=1}^{2n+1} S_i$, $I_C(f')=0$, $I_G(f')=q$ and $I_S(f')= 2n+2q$.
\end{enumerate}
\noindent 
In each case, one may alter $I_C$ and $I_S$ without altering $I_E$ and $I_G$.
\end{prop}
\begin{proof} 
Consider a canonical embedding $j$ of $S^3$ into $\mathbb{R}^4$, let $\pi\colon  S^3\longrightarrow\mathbb{R}^3$ be a canonical projection, and write  $f_0= \pi \circ j$. We have $I_E(f_0)=1$ and  $I_G(f_0)=I_C(f_0)= I_S(f_0)=0$. Let us prove ($a$). We construct the map $f$ from the map $f_1$ as described in Table \ref{figdemo}. The map $f_1$ is obtained from the map $f_0$ such that the path joining them crosses $m-1$ times by the transition  $L$ in positive direction, so that $(I_E,I_C,I_G,I_S)(f_1)=(m,m-1,0,0)$, and we have $m-1$ new singular surfaces, each one having a cuspidal curve. Finally, $f$ is obtained from the map $f_1$ such that the path joining them crosses $q$ times the transition $P_g$ in negative direction, through the $m-1$ surfaces created by the previous transitions. Thus, we have $(I_E,I_C,I_G,I_S)(f)=(m,q+m-1,0,0)$. To prove
($b$), we consider again $f_0$ as initial map in order to construct a map $f'$ through a path crossing the maps $f_1,f_2,f_3$ and $f_4$ (see Table \ref{figdemo2}). Indeed, the map $f_1$ is obtained from the map $f_0$, such that the path joining them crosses $n$ times the transition $L$ in positive direction, so that $(I_E,I_C,I_G,I_S)(f_1)=(n+1,n,0,0)$. Hence, $f_2$ is obtained from $f_1$, such that the path joining them crosses $n$ times the transition  $P_g$ in positive direction, so that $(I_E,I_C,I_G,I_S)(f_2)=(2n+1,0,0,0)$. Now, $f_3$ is obtained from $f_2$, such that the path joining them crosses $n$ times the transition  $A_3^e$ in positive direction, so that $(I_E,I_C,I_G,I_S)(f_3)=(2n+1,0,0,2n)$. Hence,  $f_4$  is obtained from $f_3$, such that the path joining them crosses $q$ times the transition  $P_g$ in negative direction, so that $(I_E,I_C,I_G,I_S)(f_4)=(2n+1,0,q,2n)$. Finally, $f'$  is obtained from $f_4$, such that the path joining them crosses $q$ times the  transition  $A^h_{3,c}$ in negative direction, so that $(I_E,I_C,I_G,I_S)(f')=(2n+1,0,q,2n+2q)$. This completes the proof.
\end{proof}
\begin{table}[htp]
\vspace{-0.3cm}
$$ \epsfxsize=12cm \epsfbox{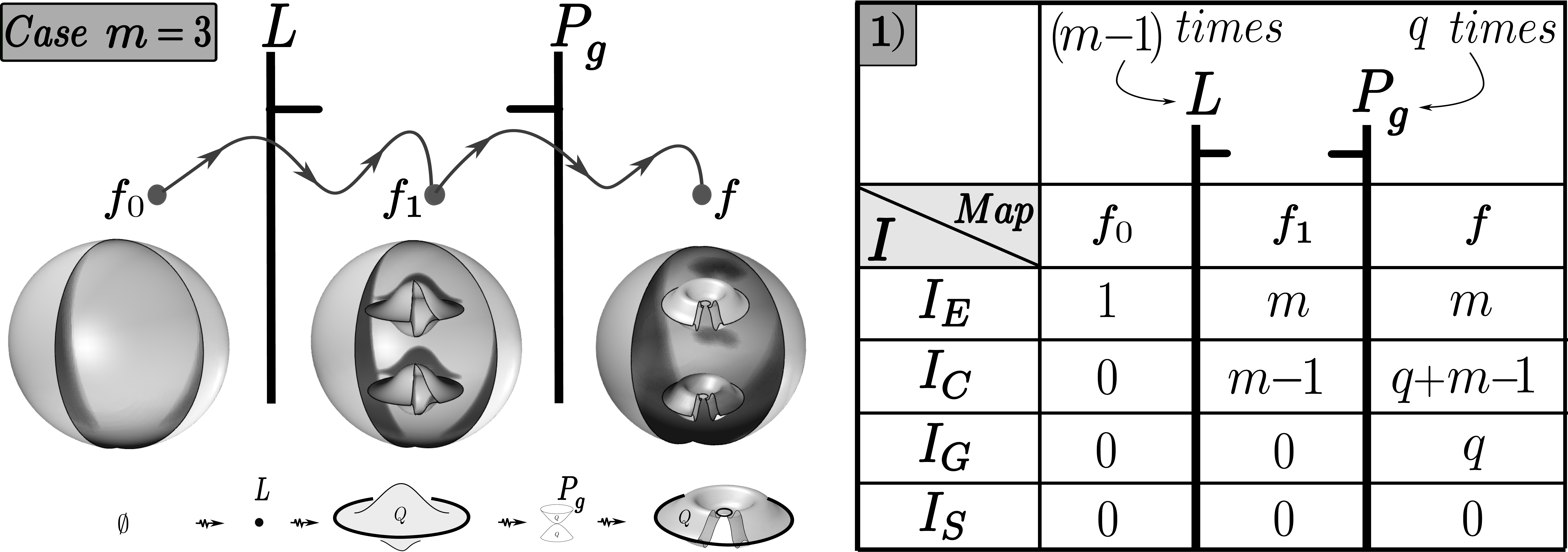} $$
\vspace{-0.2cm}
\caption{Summary of the construction of $f$ of item ($a$) of Proposition \ref{propA}.\label{figdemo} }
\end{table}
\begin{table}[htp]
\vspace{-1.1cm}
$$ \epsfxsize=11.3cm \epsfbox{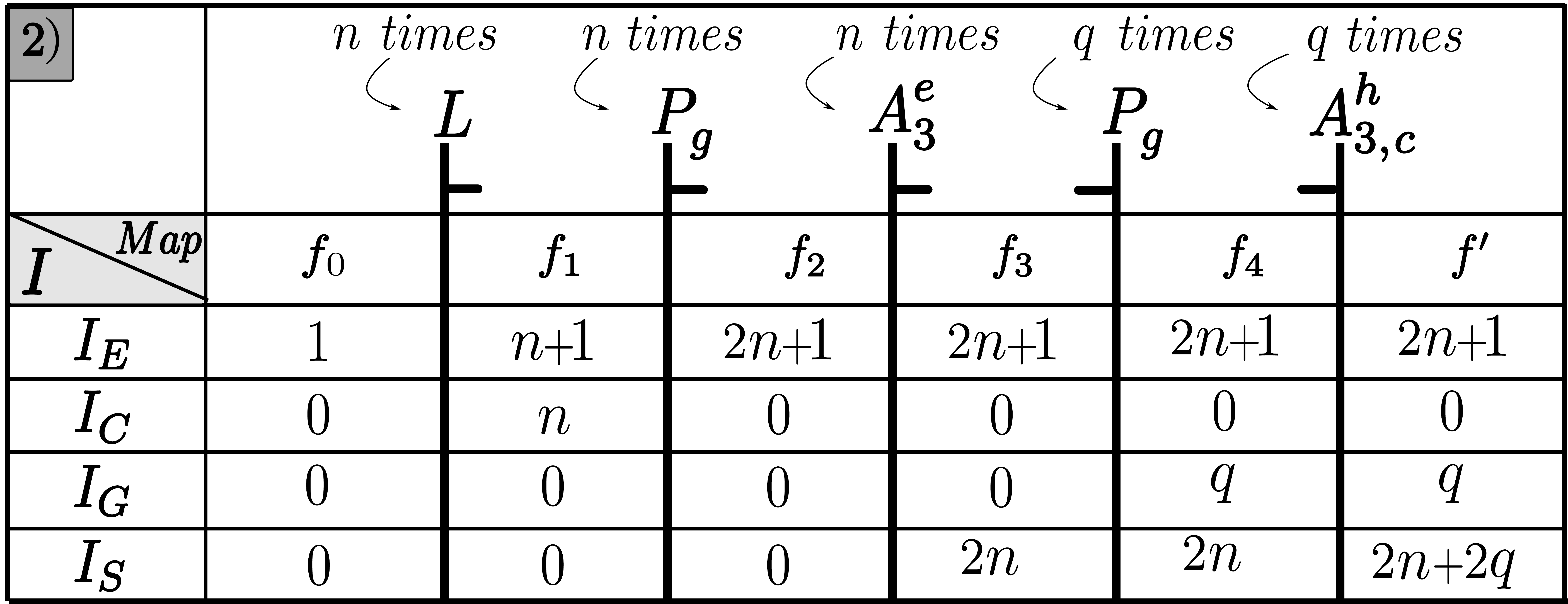} $$
\vspace{-0.2cm}
\caption{Summary of the construction of $f'$ of item ($b$) of Proposition \ref{propA}.\label{figdemo2} }
\end{table}
\vspace{-0.2cm}
\section{Relations between the global invariants}
\label{sec4}
The next result (Theorem \ref{teoA}) determines the increments of the global invariants $I_C,I_S,I_E$ and $I_G$ depending on the decompositions of the local transitions producing the set $\mathcal{T}$ given in \eqref{transall}.
\begin{definition}
Let $T$ be a transition in $\mathcal{T}$ and let us consider a path $F_t$ joining two stable maps.  We say that the {\em local increment} of $T$ is $+1$ (resp. $-1$) if the path $F_t$ passes through $T$ in positive(resp. negative) direction. 
The {\em global increment} of $T$ is the sum of all local increments of $T$. 
We write $$V(\mathcal{T})= \{\ell,b_{-,g},b_{0,g},b_{+,g},b_v,p_g, p_v,a^e_3,a^h_{3,2c},a^h_{3,c},a^h_{3,0}\}$$
for the set of increments corresponding to the transitions in $\mathcal{T}$ with respect to a path joining two stable maps.
\end{definition}

\begin{lemma}\label{Lem1} Let $f_0,f\in\mathcal{E} (S^3,\mathbb{R}^3)$. Then, the increment of $I_E$, $I_C$, $I_G$ and $I_S$ along a path transverse to the transitions of codimension $1$ in $\mathcal{T}$ are given by:
\begin{itemize}
\item[] $\Delta I_E=\ell + b_{v} +p_{v}$,
\item[] $\Delta I_C=\ell-b_{-,g}+b_{+,g}+b_v-p_g-p_v-2a^h_{3,2c}-a^h_{3,c}$,
\item[] $\Delta I_G=-b_{-,g}-b_{0,g}-b_{+,g}-p_g$,
\item[] $\Delta I_S=2(a^e_3+a^h_{3,2c}+a^h_{3,c}+a^h_{3,0})$,
\end{itemize}
\end{lemma}
\begin{proof}
From Table \ref{tablaLBPAeh}, $\Delta I_E$ increases by one when it passes through one of the transitions $L$, $B_v$ and $P_v$, but it does not increase at all when it passes through the other transitions. Thus,
$$\Delta I_E=\ell+b_{v}+p_{v}\,.$$
Similarly, from Table \ref{tablaLBPAeh} one verifies the other equalities.
\end{proof}
\begin{theorem}\label{teoA}
If $f\colon S^3\longrightarrow\mathbb{R}^3$ is a stable map, then the invariants $I_E,I_C,I_G$ and $I_S$ 
satisfy the following equality:
\begin{equation}\label{eq00}
 I_E(f)+I_G(f)+I_C(f)+I_S(f)=1+2(\ell+b_{v}-b_{-,g}-p_g+a_{3}^{e}+a_{3,0}^{h})+a_{3,c}^{h}-b_{0,g}.
\end{equation}
\end{theorem}
\begin{proof} Let us consider the map $f_0$ as in the proof of Proposition \ref{propA}. If $f\colon S^3\longrightarrow\mathbb{R}^3$ is a stable map, 
then  $f$ may be obtained from $f_0$ through a path passing only by codimension-one transitions. Since only the transitions $L,B,P,A^e_3,A^h_3$ and their  subdivisions alter the invariants $I_E,I_C,I_G$ and $I_S$, we have
\begin{displaymath}
\begin{array}{ll}
 I_E(f)=I_E(f_0)+\Delta I_E=1+\Delta I_E,\;\;\;\;\;\;\;\;\;\;\;& I_G(f)=I_G(f_0)+\Delta I_G=\Delta I_G,\\
 I_C(f)=I_C(f_0)+\Delta I_C=\Delta I_C,\;\;\;\;\;\;\;\;\;\;\;\;\;\;\;\;\;
&I_S(f)=I_S(f_0)+\Delta I_S=\Delta I_S.
\end{array}
\end{displaymath}
From these equalities, we get
$$I_E(f)+I_C(f)+I_G(f)+I_S(f)=1+\Delta I_E+\Delta I_C+\Delta I_G+\Delta I_S\,,$$
and from Lemma \ref{Lem1} we obtain the equality \eqref{eq00}.
\end{proof}
\begin{example}\label{Ej}
Let us consider the map $f$ of Figure \ref{transespecial} and its two swallowtails, say $s_1$ and $s_2$. These swallowtails are created and then eliminated only through the juxtaposition of paths among the nine paths $F_{t}^i$, $i=1,\dots,9$, where:

\begin{itemize}
\item[($i$)] $F_{t}^1$ starts at $\omega_2$, then passes through the transition  $A^h_{3,c}$ in positive direction and reaches $f$, thus creating the swallowtails $s_1$ and $s_2$. Hence, $F_{t}^1$ passes through $A^h_{3,c}$ in negative direction and reaches $f_1$, ($f_2$ or $\omega_2$), thus eliminating the swallowtails $s_1$ and $s_2$. In this case, one has $a_{3,c}^{h}-b_{0,g}=(1-1)-0=0$. 

\item[($ii$)] $F_{t}^2$ starts at $\omega_2$, then passes through the transition  $A^h_{3,c}$ in positive direction and reaches $f$, thus creating the swallowtails $s_1$ and $s_2$. Hence, $F_{t}^2$ passes through $B_{0,g}$ in positive direction to reach $f_3$, then it passes through $A^h_{3,0}$ in negative direction and reaches $f_4$, thus eliminating the swallowtails $s_1$ and $s_2$. In this case, one has $a_{3,c}^{h}-b_{0,g}=+1-(+1)=0$.  

\item[($iii$)] $F_{t}^3$ starts at $\omega_2$, then passes through the transition  $A^h_{3,c}$ in positive direction and reaches $f$, thus creating the swallowtails $s_1$ and $s_2$. Hence, $F_{t}^3$ passes through $B_{0,g}$ in negative direction and reaches  $f_5$, then it passes through $A^h_{3,0}$ in negative direction to reach $f_6$, thus eliminating the swallowtails $s_1$ and $s_2$. In this case, one has $a_{3,c}^{h}-b_{0,g}=+1-(-1)=2$.  

\item[($iv$)]  $F_{t}^4$ starts at $f_3$, then passes through the transition  $B_{0,g}$ in negative direction and reaches $f$, thus creating the swallowtails $s_1$ and $s_2$. Hence, $F_{t}^4$ passes through $B_{0,g}$ in positive direction to come back to $f_3$, then it passes through $A^h_{3,2c}$ in negative direction to reach $f_4$, thus eliminating the swallowtails $s_1$ and $s_2$. In this case, one has $a_{3,c}^{h}-b_{0,g}=0-(-1+1)=0$. 
\begin{figure}[htp]
\vspace{-0.4cm}
$$ \epsfxsize=11.4cm \epsfbox{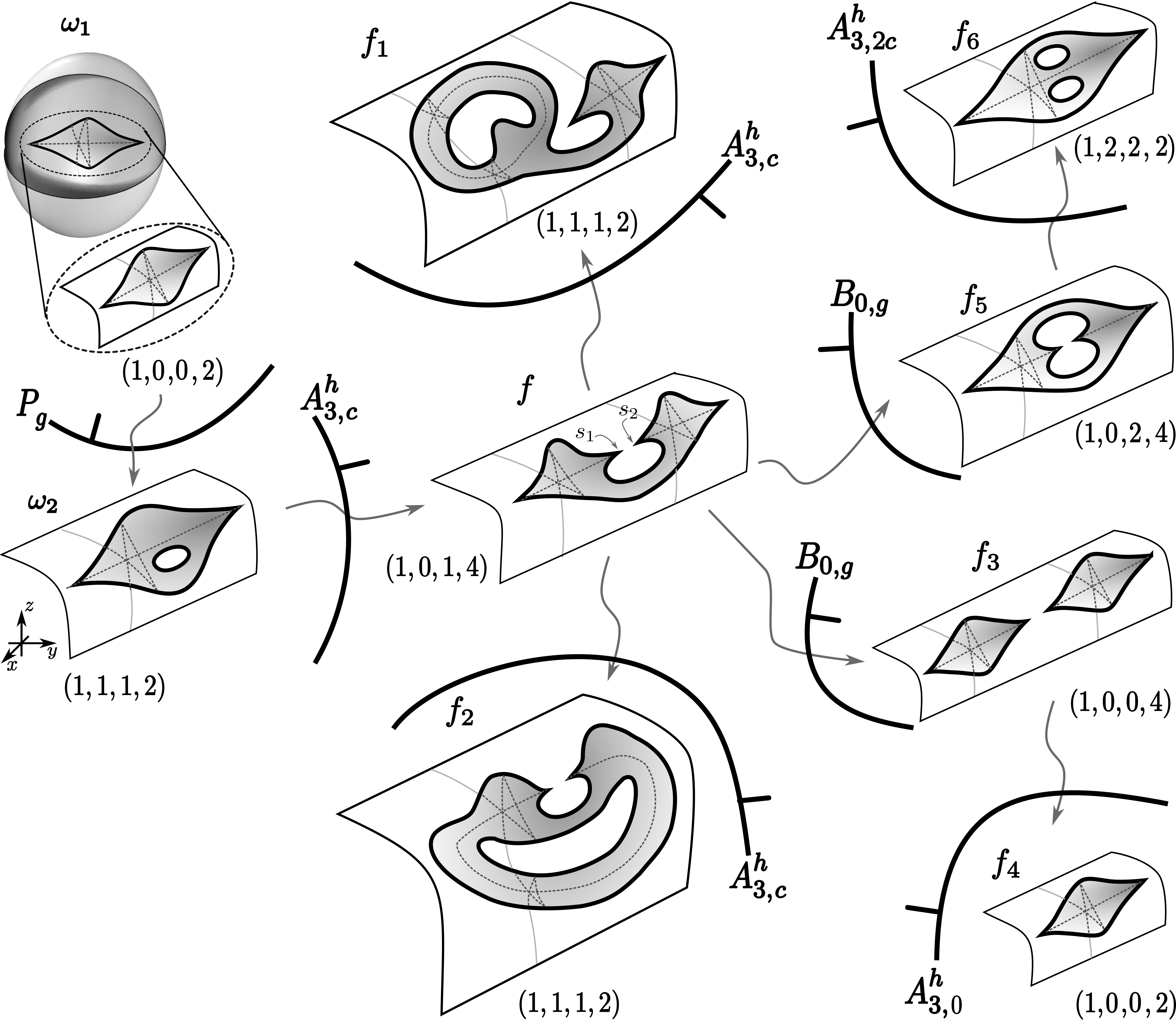} $$
\vspace{-0.8cm}
\caption{Relationship between the transicions $A_{3,c}^{h}$ and $B_{0,g}$. \label{transespecial}}
\end{figure}
\vspace{-0.15cm}
\item[($v$)] $F_{t}^5$ starts at $f_3$, then passes through the transition  $B_{0,g}$ in negative direction and reaches $f$, thus creating the swallowtails $s_1$ and $s_2$. Hence, $F_{t}^5$ passes through $B_{0,g}$ in negative direction to reach $f_5$, then it passes through $A^h_{3,2c}$ in negative direction to reach $f_6$, thus eliminating the swallowtails $s_1$ and $s_2$. In this case, one has $a_{3,c}^{h}-b_{0,g}=0-(-1-1)=2$. 
\item[($vi$)] $F_{t}^6$ starts at $f_3$, then passes through the transition  $B_{0,g}$ in negative direction and reaches $f$, thus creating the swallowtails $s_1$ and $s_2$. Hence, $F_{t}^6$ passes through $A^h_{3,c}$ in negative direction to reach $f_1$, ($f_2$ or $\omega_2$), thus eliminating the swallowtails $s_1$ and $s_2$. In this case, one has $a_{3,c}^{h}-b_{0,g}=-1-(-1)=0$. 
\item[($vii$)]  $F_{t}^7$ starts at $f_5$, then passes through the transition  $B_{0,g}$ in positive direction and reaches $f$, thus creating the swallowtails $s_1$ and $s_2$. Hence, $F_{t}^7$ passes through $B_{0,g}$ in negative direction to reach $f_5$, then it passes through $A^h_{3,2c}$ in negative direction to reach $f_6$, thus eliminating the swallowtails $s_1$ and $s_2$. In this case, one has $a_{3,c}^{h}-b_{0,g}=0-(1-1)=0$. 

\item[($viii$)] $F_{t}^8$ starts at $f_5$, then passes through the transition  $B_{0,g}$ in positive direction and reaches $f$, thus creating the swallowtails $s_1$ and $s_2$. Hence, $F_{t}^8$ passes through $B_{0,g}$ in negative direction to reach $f_3$, then it passes through $A^h_{3,0}$ in negative direction to reach $f_4$, thus eliminating the swallowtails $s_1$ and $s_2$. In this case, one has $a_{3,c}^{h}-b_{0,g}=0-(1+1)=-2$.  

\item[($ix$)]  $F_{t}^9$ starts at $f_5$, then passes through the transition  $B_{0,g}$ in positive direction and reaches $f$, thus creating the swallowtails $s_1$ and $s_2$. Hence, $F_{t}^9$ passes through $A^h_{3,c}$ in negative direction to reach $f_1$ ($f_2$ or $\omega_2$), thus eliminating the swallowtails $s_1$ and $s_2$. In this case, one has $a_{3,c}^{h}-b_{0,g}=-1-(+1)=-2$.  
\end{itemize}  
\end{example}
\begin{remark} 
In the summary of Example \ref{Ej}, it is worth noting that a path passing through $A^h_{3,c}$ in positive direction eliminates all swallowtails, if it passes through  $A^h_{3,c}$ in negative direction or it passes through $B_{0,g}$ in positive (or negative) direction. Similarly, a path passing through $B_{0,g}$ in positive (resp. negative) direction eliminates all swallowtails; if it passes through  $B_{0,g}$ in negative (resp. positive) direction, or it passes through $A^h_{3,c}$ in negative direction.
\end{remark}
\begin{lemma}\label{Lem2} Let $f_0$ be the canonical projection of $S^3$ onto $\mathbb{R}^3$, and let $f\in\mathcal{E} (S^3,\mathbb{R}^3)$. If  $I_S(f)=0$ and $\gamma$ is a path joining $f_0$ and $f$, then
$$a_{3,c}^{h}-b_{0,g}=0\hspace{-0.18cm}\mod 2\,.$$
\end{lemma}
\begin{proof}
Since $I_S(f)=0$,  $f$ has no swallowtails. Hence, $f$ may be obtained by a path $\gamma$ composed of paths among the nine paths of Example \ref{Ej}. In each case, we have $a_{3,c}^{h}-b_{0,g}=0 \mod 2$. Now, if $\gamma$ does not contain any such a path, then $a_{3,c}^{h}-b_{0,g}=0-0=0 \mod 2$. 
\end{proof}
\begin{cor}\label{corP}
Let $f\colon S^3\longrightarrow \mathbb{R}^3$ be a stable map.
\begin{enumerate}
\item[($a$)]  If  $I_S(f)=0$, then  
$I_E(f)+I_G(f)+I_C(f)=1\hspace{-0.18cm}\mod 2$.
\item[($b$)] If $f$ is a fold map, then $I_C(f)=I_S(f)=0$ and
$I_E(f)+I_G(f)=1\hspace{-0.18cm}\mod 2$.
\end{enumerate}
\end{cor}
\begin{proof} If $I_S(f)=0,$ then its branch set has no swallowtails and by Lemma \ref{Lem2}, we have $a_{3}^{e}-b_{0,g}=0\mod 2$. Now, by Theorem \ref{teoA} we get
\begin{displaymath}
\begin{array}{ll}
I_E(f)+I_G(f)+I_C(f)&=1+2(\ell+b_{v}-b_{-,g}-p_g+a_{3}^{e}+a_{3,0}^{h})+0\hspace{-0.18cm}\mod 2,\\
&=1\hspace{-0.18cm}\mod 2.
\end{array}
\end{displaymath}
This proves ($a$). Let us now prove ($b$). If  $f$ is a fold map, then its branch set has no cuspidal curves   nor swallowtails, in other words $I_C(f)=I_S(f)=0$, thus replacing this in part ($a$) we have immediately,
\vspace{-0.2cm}
\begin{displaymath}
I_E(f)+I_G(f)=1\hspace{-0.25cm}\mod 2.
\vspace{-0.5cm}
\end{displaymath}
\end{proof}
Let $\{S_i\}_{i=1}^n$ be a collection of $n$ compact, oriented surfaces embedded in $\mathbb{R}^3$. A natural question arises: 

\noindent What are the conditions on the collection $\{S_i\}_{i=1}^n$ in order to form a branch set of a fold map $f\colon S^3\longrightarrow \mathbb{R}^3$?

One necessary condition on the collection $\{S_i\}_{i=1}^n$ is the item ($b$) of Corollary \ref{corP}. In this case, we have $I_E=n$ and $I_G=\sum_{i=1}^ng(S_i)$, hence we get the formula:
\begin{displaymath}
n+\sum_{i=1}^ng(S_i)=1\hspace{-0.2cm}\mod 2\,.
\end{displaymath}
\begin{example}\label{Exp}
Let $\{S_i\}_{i=1}^n$ be a collection of spheres centred in the origin of $\mathbb{R}^3$ with $n$ odd, such that the radio of each $S_i$ is $r_i=i$ for all $i=1,\dots,n$. Then, there is a map $f\colon S^3\longrightarrow \mathbb{R}^3$ such that its branch set coincides with $\{S_i\}_{i=1}^n$, that is $\Sigma f=\bigcup_{i=1}^nS_i$, and the surfaces with radio less or equal than $(n-1)/2$ have outer direction and the surfaces with radio greater than $(n-1)/2$ have inner direction. Figure \ref{apliS3fin1} shows the construction for the case $n=3$.
\end{example}
\begin{figure}[htp]
\vspace{-0.6cm}
$$ \epsfxsize=11.5cm \epsfbox{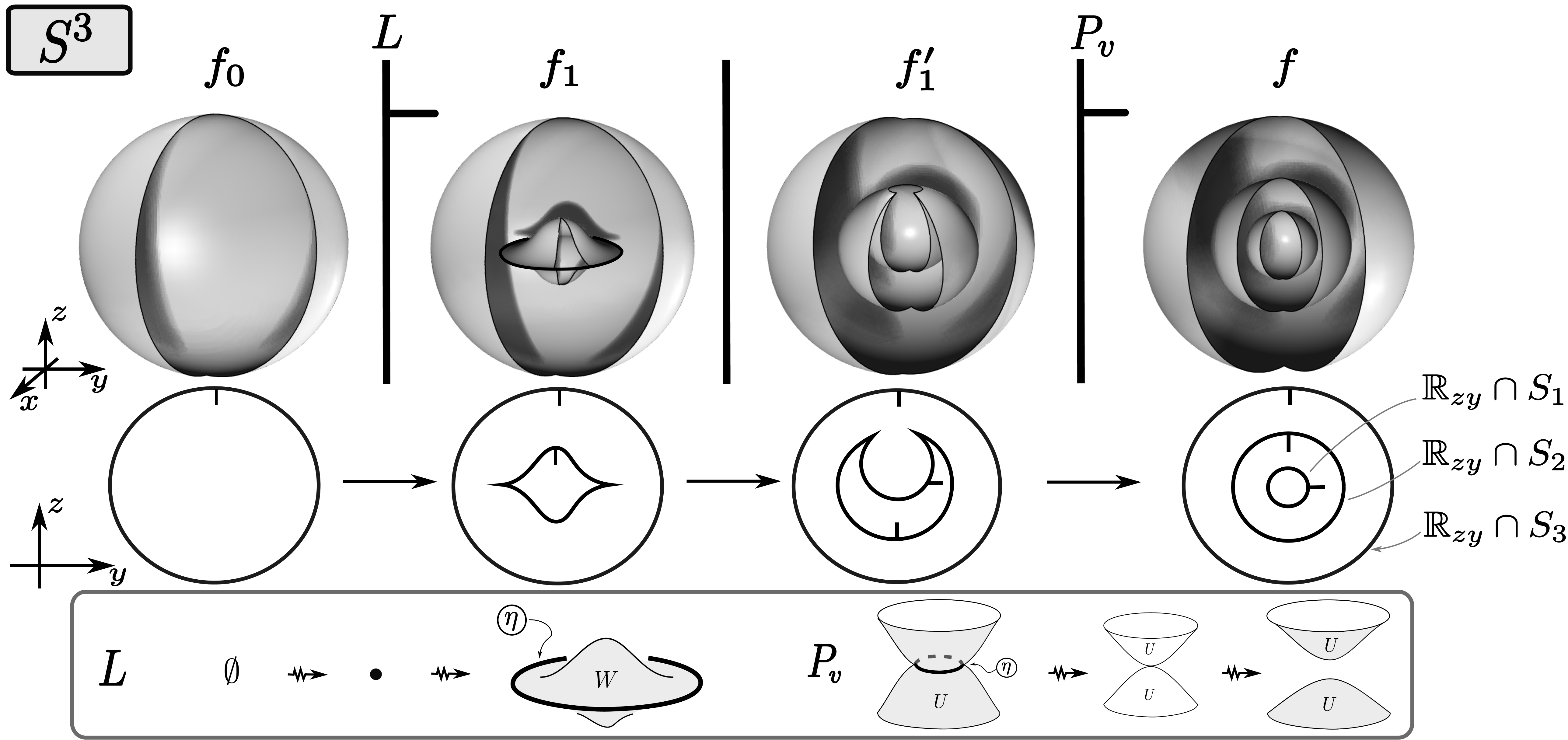} $$
\vspace{-0.8cm}
\caption{Example \ref{Exp}, case $n=3$. \label{apliS3fin1}}
\end{figure}
\begin{example} \label{Exp2}
  Let $\{S_i\}_{i=1}^n$ be a collection of surfaces embedded in $\mathbb{R}^3$ with $n$ odd, such that,
\begin{itemize}
\item[1.] the surface $S_1$ is a sphere embedded in $\mathbb{R}^3$;
\item[2.] the surfaces $\{S_i\}_{i=2}^n$ are disjoint and enclosed by $S_1$, and  the pairs $\{S_{2i}$, $S_{2i+1}\}$ are surfaces one inside the other, and both with genus $k_i\in \mathbb{Z}^+$, for all $i=1,\dots,(n-1)/2$. 
\end{itemize} 
Then, there is a map $f\colon S^3\longrightarrow \mathbb{R}^3$ such that its branch set coincides with $\{S_i\}_{i=1}^n$, that is $\Sigma f=\bigcup_{i=1}^nS_i$. Figure \ref{apliS3fin2} shows the construction for the case $n=3$ and $k_2=2$.
\begin{figure}[htp]
\vspace{-0.4cm}
$$ \epsfxsize=12cm \epsfbox{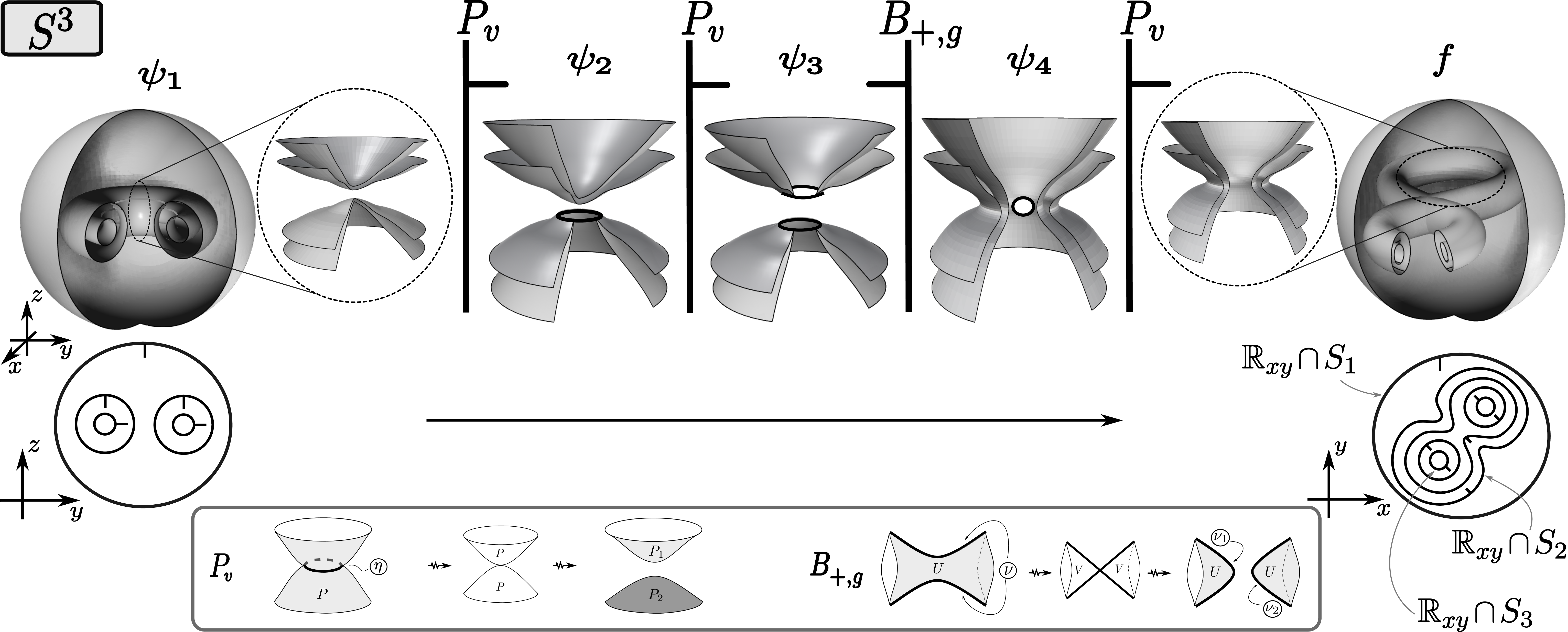} $$
\vspace{-0.8cm}
\caption{Example \ref{Exp2}, case $n=3$. \label{apliS3fin2}}
\end{figure}
\end{example}
In Figure \ref{foldmaps2}, we have some branch sets of fold maps of $S^3$ in $\mathbb{R}^3$ with $I_E=5$ and $I_G\leq 6$, that can be constructed in a similar way as the maps are constructed in Example \ref{Exp} and \ref{Exp2}.
\begin{figure}[htp]
\vspace{-0.3cm}
$$ \epsfxsize=10.7cm \epsfbox{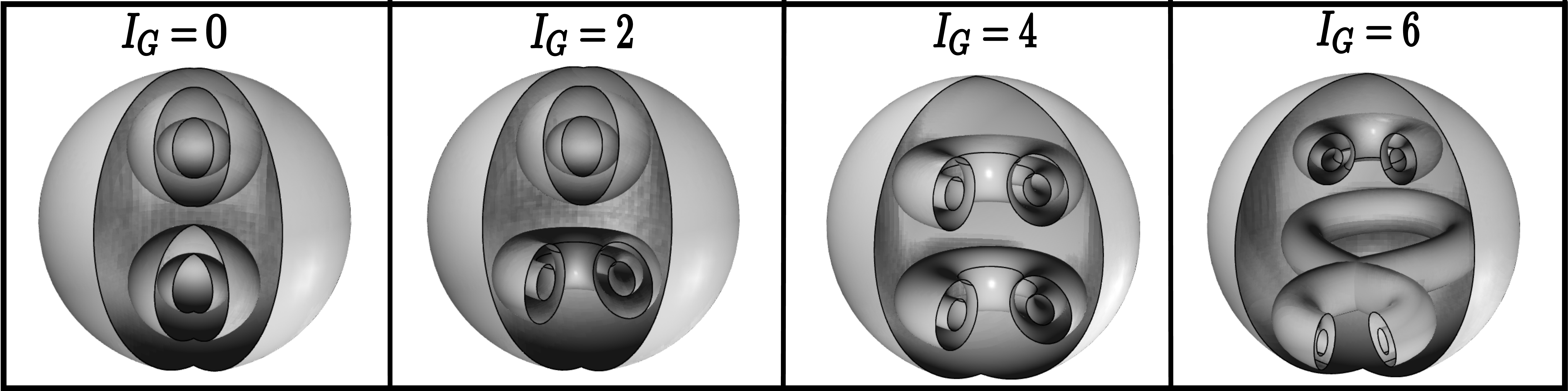} $$
\vspace{-0.8cm}
\caption{Fold maps from $S^3$ to $\mathbb{R}^3$ with $I_E=5$ and $I_G\leq 6$. \label{foldmaps2}}
\end{figure}
\bibliographystyle{amsplain}

\end{document}